\documentclass[reqno,10pt, centertags]{amsart}
\usepackage{amsmath,amsthm,amscd,latexsym}
\usepackage{amssymb,upref,esint,dsfont,cite,amsmath}

\usepackage{hyperref}
\newcommand*{\mailto}[1]{\href{mailto:#1}{\nolinkurl{#1}}}

\makeatletter
\def\theequation{\@arabic\c@equation}

\newcommand{\diag}{\operatorname{diag}}

\newcommand{\bbN}{{\mathbb{N}}}
\newcommand{\bbR}{{\mathbb{R}}}

\newcommand{\bbC}{{\mathbb{C}}}

\newcommand{\cA}{{\mathcal A}}
\newcommand{\cB}{{\mathcal B}}

\newcommand{\cF}{{\mathcal F}}
\newcommand{\cH}{{\mathcal H}}
\newcommand{\cI}{{\mathcal I}}

\newcommand{\cK}{{\mathcal K}}
\newcommand{\cL}{{\mathcal L}}

\newcommand{\dott}{\,\cdot\,}

\newcommand{\no}{\nonumber}
\newcommand{\lb}{\label}
\newcommand{\f}{\frac}

\newcommand{\ol}{\overline}

\newcommand{\Oh}{O}

\newcommand{\la}{\lambda}

\newcommand{\Om}{\Omega}

\newcommand{\supp}{\text{\rm{supp}}}

\newcommand{\bi}{\bibitem}

\renewcommand{\Re}{\text{\rm Re}}
\renewcommand{\Im}{\text{\rm Im}}

\renewcommand{\diag}{\text{\rm diag}}
\renewcommand{\max}{\text{\rm max}}
\renewcommand{\min}{\text{\rm min}}

\newcommand{\eT}{T}
\def\nn{\nonumber}
\def\a{\alpha}
\def\b{\beta}
\def\g{\gamma}
\def\G{\Gamma}

\def\Lam{\Lambda}

\def\la{\lambda}
\def\om{\omega}
\def\Om{\Omega}

\def\Up{\Upsilon}
\def\vp{\varphi}

\def\ve{\varepsilon}
\def\wh{\widehat}
\def\wt{\widetilde}
\def\ov{\overline}

\def\p{\partial}

\def\BC{{\mathbb C}}
\def\BR{{\mathbb R}}

\def\cla{{\mathcal A}}

\def\clf{{\mathcal F}}

\def\diag{\mathrm{diag}}
\numberwithin{equation}{section}

\newtheorem{theorem}{Theorem}[section]
\newtheorem{lemma}[theorem]{Lemma}
\newtheorem{corollary}[theorem]{Corollary}
\newtheorem{proposition}[theorem]{Proposition}
\newtheorem{definition}[theorem]{Definition}

\theoremstyle{remark}
\newtheorem{remark}[theorem]{Remark}
\newtheorem{notation}[theorem]{Notation}

\begin{document}

\title[The Inverse Approach to Dirac-Type Systems]{The Inverse Approach to Dirac-Type Systems \\
Based on the $A$-Function Concept}
%{The $A$-Equation for Dirac-Type Operators}

\author[F.\ Gesztesy]{Fritz Gesztesy}
\address{Department of Mathematics,
Baylor University, One Bear Place \#97328,
Waco, TX 76798-7328, USA}
\email{\mailto{Fritz\_Gesztesy@baylor.edu}}
%\email{Fritz$\_$Gesztesy@baylor.edu}
\urladdr{\url{http://www.baylor.edu/math/index.php?id=935340}}
%\urladdr{http://www.baylor.edu/math/index.php?id=935340}

\author[A.\ L.\ Sakhnovich]{Alexander Sakhnovich}
\address{Faculty of Mathematics\\ University of Vienna\\
Oskar-Morgenstern-Platz 1\\ 1090 Wien\\ Austria}
\email{\mailto{oleksandr.sakhnovych@univie.ac.at}}
%\email{oleksandr.sakhnovych@univie.ac.at}
\urladdr{\url{http://www.mat.univie.ac.at/~sakhnov/}} 
%\urladdr{http://www.mat.univie.ac.at/~sakhnov/}

%%%%%%%%%%%%%%%%%%%%%%%%%%%%%%%%%%%%%%%%% 
%\dedicatory{}
\date{\today}
%\date{, 2003.}
\thanks{The research of A.\ L.\ Sakhnovich was supported by the
Austrian Science Fund (FWF) under Grant No.~P29177.} 
%\thanks{To appear in {\it }, .}

%%%%%%%%%%%%%%%%%%%%%%%%%%%%%%%%%%%%%%%%% 
\begin{abstract}
The principal objective in this paper is a new inverse approach to general Dirac-type systems of the form 
\begin{align*} &      
y^{\prime}(x, z )=i (z J + J V(x))y(x,
z ) \quad (x \geq 0),
\end{align*} 
where $y = (y_1,\dots,y_m)^{\top}$ and (for $m_1, m_2 \in \bbN$) 
\begin{align*} &    
J = \begin{bmatrix}
I_{m_1} & 0_{m_1 \times m_2} \\ 0_{m_2 \times m_1} & -I_{m_2}
\end{bmatrix}, \quad  V= \begin{bmatrix}
0_{m_1} & v \\ v^{*} & 0_{m_2} \end{bmatrix},  \quad m_1+m_2=:m,  
 \end{align*}  
for $v \in \big[C^1([0, \infty))\big]^{m_1 \times m_2}$, 
modeled after B.~Simon's 1999 inverse approach to half-line Schr\"odinger operators. 
In particular, we derive the $\cla$-equation associated to this Dirac-type system in the 
($z$-independent) form
\[
\frac{\p}{\p \ell}\cla(x, \ell) = \frac{\p}{\p x}\cla(x,\ell)
+ \int_0^x \cla(x-t,\ell)\cla(0,\ell)^*\cla(t,\ell) \, dt \quad (x \geq 0, \ell \geq 0).
\]
Given the fundamental positivity condition $S_T > 0$ in \eqref{053S} (cf.\ \eqref{053} for 
details), we prove that this integro-differential equation for 
$\cla(\dott,\dott)$ is uniquely solvable for initial conditions   
\[ 
\cla(\dott,0)=\cla(\dott) \in \big[C^1([0,\infty))\big]^{m_2 \times m_1}, 
\]
and the corresponding potential coefficient $v \in \big[C^1([0,\infty))\big]^{m_1 \times m_2}$ can be 
recovered from $\cla(\dott,\dott)$ via
\[
v(\ell) = - i \cla(0,\ell)^* \quad (\ell \geq 0).
\] 
\end{abstract}
%%%%%%%%%%%%%%%%%%%%%%%%%%%%%%%%%%%%%%%%% 

\maketitle

%\newpage 

{\scriptsize{\tableofcontents}}
%\normalsize

%%%%%%%%%%%%%%%%%%%%%%%%%%%%%%%%%%%%%%%%%
%%%%%%%%%%%%%%%%%%%%%%%%%%%%%%%%%%%%%%%%%
\section{Introduction} \lb{s1}
%%%%%%%%%%%%%%%%%%%%%%%%%%%%%%%%%%%%%%%%%
%%%%%%%%%%%%%%%%%%%%%%%%%%%%%%%%%%%%%%%%%

Barry Simon's seminal paper \cite{Si99} on a new approach to inverse spectral problems for scalar Schr\"odinger operators $(-d^2/dx^2) + V(x)$ on the half-line $(0,\infty)$ with a Dirichlet boundary condition at $x=0$ (for simplicity), is based on the so-called $A$-equation
\begin{align}& \lb{A1}
\frac{\p A}{\p x}(\alpha,x)= \frac{\p A}{\p \a}(\alpha,x) + \int_0^{\a}A(\b,x)A(\a - \b, x) \, d \b \quad 
(x \geq 0, \alpha \geq 0). 
\end{align}
The $A$-function $A(\alpha,x)$ on the other hand is connected with the Weyl--Titchmarsh $m$-function $m(z,x)$ of the Schr\"odinger equation on the half-line $[x, \infty)$ ($x > 0$) via the Laplace (or Fourier) transformation in the complex domain. More precisely, assuming $V\in L^1((0,\eT))$ for all $\eT \in (0,\infty)$ to be 
real-valued, one has 
\begin{align} 
\begin{split} 
m(z,x)= i z^{1/2}- \int_0^{\eT}A(\a,x)e^{2i \a z^{1/2}} \, d\a + \Oh\Big(e^{i(2 \eT-\ve) z^{1/2}}\Big),& \lb{A1+} \\ 
z \in \bbC\backslash [0,\infty), \; \pi - \varepsilon < \arg\big(i z^{1/2}\big) < (\pi/2) + \varepsilon,&
\end{split}
\end{align}
for any $\eT > \ve >0$. In particular, one notes that \eqref{A1} is independent of the spectral parameter $z$. 

Introducing the set 
\begin{equation}
\boldsymbol{A}_{\eT}=\{A\in L^1([0,\eT])\,|\, \text{$A$ real-valued}, \, I+\cK_A >0 \},   \lb{0.23}
\end{equation}
where
\begin{align}
& (\cK_A f)(\alpha)=\int_0^{\eT} K(\alpha,\beta) f(\beta) \, d\beta, \quad 
\alpha \in [0,\eT], \; f\in L^2((0,\eT)), \\
& K(\alpha, \beta)=[\phi(\alpha-\beta)-\phi(\alpha+\beta)]/2, \quad 
\phi(\alpha)=\int_0^{|\alpha|/2} A(\gamma) \, d\gamma, \quad (\alpha, \beta \in [0,\eT]),   \no 
\end{align}
Remling \cite{Re03} completed the inverse approach based on the $A$-function by proving that 
$\boldsymbol{A}_{\eT}$ is precisely the set of $A$-functions in 
\begin{align} 
\begin{split} 
m(z)= i z^{1/2}- \int_0^{\eT}A(\a)e^{2i \a z^{1/2}} \, d\a+ \Oh\Big(e^{i(2 \eT-\ve) z^{1/2}}\Big),& \lb{A1+a} \\ 
z \in \bbC\backslash [0,\infty), \; \pi - \varepsilon < \arg\big(i z^{1/2}\big) < (\pi/2) + \varepsilon,&
\end{split}
\end{align} 
for all $\eT > 0$. Equivalently, given $A_0\in L^1((0,\eT))$, there exists a potential 
$V\in L^1((0,\eT))$ such that $A_0$ is the $A$-function of $V$ if and only if 
$A_0\in \boldsymbol{A}_{\eT}$. In particular, we emphasize that the $A$-equation \eqref{A1} plays a crucial role in Simon's new inverse approach to half-line Schr\"odinger operators (in contrast to the traditional approach in \cite{GL55}, \cite{Le87}--\cite{LS75}, \cite{Ma11, Re02}), since 
\[
V(x) = A(0,x) \, \text{ for a.e.~$x \in [0,\eT]$.} 
\]

Additional work on the $A$-function ensued in \cite{GS00}, and other interesting results by different authors followed (see, e.g., 
\cite{Hi01, Re03, Zh03, Zh06}, and the survey article \cite{Ge07}). 

Although Dirac-type systems present a natural generalization of the Schr\"odinger equation and there are
many analogies between the corresponding spectral and inverse spectral theories, the analog of the $A$-equation  for Dirac-type system 
(which will be called
the $\cla$-equation to distinguish it from the Schr\"odinger equation case)
was not formulated until now. In the present paper we fill this gap and formulate and also prove {\it solvability} of the $\cla$-equation for the general Dirac-type system (in the general matrix-valued case)
\begin{align} &       \lb{A2}
y^{\prime}(x, z )=i (z J + J V(x))y(x,
z ), \quad
x \geq 0 \quad \left(y^{\prime}(x,z):=\frac{dy}{d x}(x,z)\right),
\end{align} 
where $y = (y_1,\dots,y_m)^{\top}$ and (for $m_1, m_2 \in \bbN$) 
\begin{align} &   \lb{A3}
J = \begin{bmatrix}
I_{m_1} & 0_{m_1 \times m_2} \\ 0_{m_2 \times m_1} & -I_{m_2}
\end{bmatrix}, \quad  V= \begin{bmatrix}
0_{m_1} & v \\ v^{*} & 0_{m_2} \end{bmatrix},  \quad m_1+m_2=:m. 
 \end{align} 
Here $\bbN$ stands for the set of natural numbers and $I_k$ is the $k \times k$ identity matrix. The $m \times m$ matrix-valued function $V$ (and, sometimes, the $m_1 \times m_2$ matrix-valued function $v$) is called the potential of the Dirac-type system \eqref{A2}.  

Before being able to formulate some of our principal results, we need a few preparations and some standard definitions. 

%%%%%
\begin{notation} \lb{NnFS} The $m \times m$ fundamental solution of system \eqref{A2}, normalized by $I_m$ at $x=\ell$, is denoted by $u_{\ell}(x,z)$, that is, one has, 
\begin{equation}\lb{e1} 
u_{\ell}(\ell,z)=I_m.
\end{equation}
\end{notation} 
%%%%%

%%%%%
\begin{definition} \lb{DnWF} 
Suppose that $v \in \big[L^1((0,R))\big]^{m_1 \times m_2}$ for all $R > 0$. Then the 
 $m_2\times m_1$ matrix function $\vp_{\ell}(z)$ is called a $($matrix-valued\,$)$ Weyl--Titchmarsh function associated with the system \eqref{A2} on $[\ell, \infty)$ $(\ell \geq 0)$
if it is holomorphic in $\BC_+$ and the entries of $u_{\ell}(\dott ,z)\begin{bmatrix}I_{m_1} \\ \vp_{\ell}(z)\end{bmatrix}$ belong to $L^2((\ell, \infty))$,
that is, 
\begin{align}\lb{e2}&
u_{\ell}(\dott ,z)\begin{bmatrix}I_{m_1} \\ \vp_{\ell}(z)\end{bmatrix} \in \big[L^2((\ell, \infty))\big]^{m \times m_1}.
\end{align}
\end{definition}
%%%%%

%%%%%
\begin{remark}\lb{RkWF}
The Weyl--Titchmarsh function $\vp_{\ell}(\dott, z)$ always exists, is unique, and contractive for $z \in \BC_+$ (see \cite{FKRS12} and \cite[Prop.~2.17 and Cor.~2.21]{SSR13}). 
\hfill  $\diamond$
\end{remark}
%%%%%

Propositions \ref{cyHea} and \ref{PnDif}, and Remark \ref{RkSdvig} in the present paper then yield the next statement:

%%%%%
\begin{proposition}\lb{PnH} Assume that  $v \in \big[C^1([0, \infty))\big]^{m_1 \times m_2}$ in the Dirac-type  system \eqref{A2}. Then, the associated Weyl--Titchmarsh functions $\vp_{\ell}(z)$ admit a representation of the type 
\begin{align} \lb{Repr0}&
\vp_{\ell}(z)=2i z\int_0^{\infty}e^{2i xz}\Phi(x, \ell) \, dx \quad (\Im (z) >0, \; \ell \geq 0),
\end{align} 
where
\begin{equation} 
\Phi(\dott, \ell) \in \big[C^2([0, \infty))\big]^{m_2 \times m_1} \quad (\ell \geq 0).
\end{equation} 
\end{proposition}
%%%%%

The analog of the $A$-function  in \eqref{A1+} for the Dirac-type system \eqref{A2} (which will be called the $\cla$-function) can now be introduced via  
 \begin{align} &       \lb{042}
\cla(x, \ell):=\frac{\p}{\p x}\Phi(x,\ell) \equiv \Phi^{\prime}(x,\ell); \quad \cla(x):=\cla(x,0) \quad (x \geq 0, \; \ell \geq 0).
\end{align} 
According to Proposition \ref{PnH} and Corollary \ref{CyNesCond} the following result holds.

%%%%%
\begin{proposition} \lb{PnH2} Assume that  $v \in \big[C^1([0, \infty))\big]^{m_1 \times m_2}$ in the Dirac-type  system \eqref{A2}. Then, with $\cla(\dott)$ given by \eqref{042},
the operator $S := S_{\eT}$ in $\big[L^2((0,\eT))\big]^{m_2}$ $($$\eT \in (0,\infty)$$)$ introduced as  
\begin{align}  \lb{053}
\begin{split} 
& (S f)(x) = (S_{\eT} f)(x) = f(x) - \int_0^{\eT}s(x,t) f(t) \, dt, \quad f \in \big[L^2((0,\eT))\big]^{m_2},   \\ 
& s(x,t):=\int_0^{\min(x,t)}\cla(x-\xi)\cla(t-\xi)^* \, d\xi \quad ((x,t) \in (0,\eT) \times (0,\eT)), 
\end{split} 
\end{align}
satisfies the positive definiteness property  
\begin{equation}
S = S_{\eT} > 0     \lb{053S}
\end{equation}
for all $\eT \in (0,\infty)$.
\end{proposition}
%%%%%

Now, we are in a position to formulate our principal theorem, an immediate consequence of Theorems \ref{AeqfD} and \ref{Tm6.3},
and formula \eqref{e43}:

%%%%%
\begin{theorem}\lb{MT0} 
Suppose that \, $v \in \big[C^1([0, \infty))\big]^{m_1 \times m_2}$ in the Dirac-type  system \eqref{A2}. Then, the 
$\cla(\dott,\dott)$-function given by
\eqref{042} is continuously differentiable, $\cla(\dott,\dott)\in \big[C^1(\bbR^2_{+,+})\big]^{m_2 \times m_1}$, 
with $\bbR^2_{+,+} = \big\{(x,\ell) \in \bbR^2 \, \big| \, x\geq 0, \, \ell \geq 0\big\}$, and satisfies the integro-differential 
equation $($for $(x,\ell) \in \bbR^2_{+,+}$$)$
\begin{equation}      \lb{044}
\frac{\p}{\p \ell}\cla(x, \ell) = \frac{\p}{\p x}\cla(x,\ell)
+ \int_0^x \cla(x-t,\ell)\cla(0,\ell)^*\cla(t,\ell) \, dt.
\end{equation} 
Conversely, the system \eqref{044} with initial condition
\begin{equation}    \lb{062}
\cla(\dott,0)=\cla(\dott)\in \big[C^1([0,\infty))\big]^{m_2 \times m_1},
\end{equation}
such that the positive definiteness property \eqref{053S} holds for all $\eT \in (0,\infty)$, has a unique solution 
\begin{equation} 
\cla(\dott,\dott) \in \big[C^1(\bbR^2_{+,+})\big]^{m_2 \times m_1}, 
\end{equation} 
and an associated Dirac-type system of the form \eqref{A2} exists. In particular, the potential coefficient 
$v \in \big[C^1([0,\infty))\big]^{m_1 \times m_2}$ in this Dirac-type system is recovered 
from $\cla(0,\dott)$ via 
\begin{equation}
v(\ell) = - i \cla(0,\ell)^* \quad (\ell \geq 0).    \lb{v}
\end{equation}
\end{theorem}
%%%%%%

One notes that in analogy to \eqref{A1}, \eqref{044} is also independent of the spectral parameter $z$. 

Given these facts, we can summarize the direct and inverse $\cla$-function problem as follows upon introducing the 
set 
\begin{align}
\begin{split} 
& \boldsymbol{\cla} = \big\{\cla \in [C^1([0,\infty))]^{m_2 \times m_1} \big| S:= S_{\eT} \, \text{defined in \eqref{053} 
satisfies $S_{\eT} > 0$} \\
& \hspace*{8.3cm} \text{for all $\eT \in (0,\infty)$} \big\}.   
 \end{split} 
\end{align} 
\noindent 
{\bf Direct Problem:} \\[2mm] 
\vspace{5pt} 
\fbox{\begin{minipage}{354pt}
\begin{align}
& v \in \big[C^1([0,\infty))\big]^{m_1 \times m_2} 
\longrightarrow \vp_{\ell}(z) 
\xrightarrow[inverse \, Laplace \, transform]{\text{by \eqref{Repr0}}} \Phi(x, \ell), \; (x,\ell) \in \bbR^2_{+,+} \no \\
& \quad \xrightarrow{\text{by \eqref{042}}} \cla(x,\ell) = \f{\partial}{\partial x} \Phi(x,\ell), \; \; (x,\ell) \in \bbR^2_{+,+}   
\longrightarrow \cla(\dott,0) := \cla(\dott) \in \boldsymbol{\cla}.  \lb{1.17a}
\end{align}
\end{minipage}} 
\vspace{5pt}

\noindent 
{\bf Inverse Problem:} \\[2mm] 
\vspace{5pt} 
\fbox{\begin{minipage}{354pt}
\begin{align}
\begin{split}
& \cla(\dott) \in \boldsymbol{\cla} \xrightarrow{\text{by \eqref{044}, \eqref{062}}} \cla(\dott, \dott) \in 
\big[C^1(\bbR^2_{+,+})\big]^{m_2 \times m_1}       \\
& \quad \xrightarrow{\text{by \eqref{v}}} v = - i \cla(0, \dott)^* \in \big[C^1([0,\infty))\big]^{m_1 \times m_2}.     \lb{1.17b}
\end{split} 
\end{align}
\end{minipage}} 
\vspace{10pt}

While we focus for convenience on half-line Dirac-type systems \eqref{A2} and global solutions $\cla(\dott,\dott)$ of \eqref{044} on $\bbR^2_{+,+}$, our results apply locally just as well; this fact is further discussed in Remark \ref{r6.4} and in the local Borg--Marchenko uniqueness result, Theorem \ref{t6.5}. 

Although the $\cla$-function equation is particularly interesting in connection with inverse problems, we here note 
that it also characterizes shifts (i.e., translations) of Dirac-type systems. As one of our results that is of independent interest, we also mention the series representation of the Weyl--Titchmarsh function, which was derived in the process 
of the proof of Theorem \ref{AeqfD} (see Lemma \ref{lA.3}).

%%%%%%
\begin{remark} \lb{r1.7}
$(i)$ Our results apply, in particular, to the special class of $m_1 \times m_1$ matrix-valued Schr\"odinger operators. The latter case is equivalent  to the subclass 
$m_1=m_2$ of the so called supersymmetric Dirac-type operators for which $v=-v^*$ holds.  For a detailed discussion of supersymmetric Dirac-type operators and their close (essentially, one-to-one) relations with (matrix-valued) Schr\"odinger operators via the celebrated Miura-type transformations (i.e., via appropriate Riccati-type relations) we refer, for instance, to \cite{BGGSS87, GSS91, GGHT12, EGNT14, EGNST15}. \\[1mm] 
$(ii)$ In the special scalar 
context, $m_1 = m_2 = 1$, the fact \eqref{1.17b} represents the Dirac-type analog of the Schr\"odinger operator results \eqref{0.23}--\eqref{A1+a} due to Remling \cite{Re03}.  \hfill $\diamond$
\end{remark}
%%%%%%

The precise connection between the $\cla$-function and the underlying matrix-valued spectral function, and hence the connection between $\cla$ and inverse spectral theory, is in preparation \cite{GS20}.  
 
Finally, we briefly summarize the notation used in this paper:   
$I_{r}$ and $0_r$ are the $r \times r$ identity and zero matrix in $\bbC^r$, respectively ($r \in \bbN$). Similarly, $0_{r_1 \times r_2}$ represents the (rectangular) $r_1 \times r_2$ zero matrix ($r_1, r_2 \in \bbN$).

Throughout this paper, for $X$ a given space, $A \in X^{r_1 \times r_2}$ represents an $r_1 \times r_2$ block (operator) matrix $A$ with entries in $X$ ($r_1, r_2 \in \bbN$). If $r_2 = 1$, we use the notation 
$A \in X^{r_1}$.

We use the abbreviation $L^p(\Omega) \equiv L^p(\Omega; dx)$ ($\Omega \subseteq \bbR$ measurable, 
$p \geq 1$) whenever Lebesgue measure $dx$ is understood. 

$\big[L^2((0, \ell))\big]^r$ is the class of square integrable vector functions on $(0, \ell)$
with values in $\BC^{r}$ and scalar product 
\begin{equation} 
(f,g)=\int_0^\ell g(x)^*f(x) \, dx = \sum_{s=1}^r \int_0^{\ell} \ol{g_s(x)} f_s(x) \, dx,
\end{equation} 
where $f=(f_1, \dots , f_r)^{\top}, g=(g_1,\dots , g_r)^{\top} \in \big[L^2((0, \ell))\big]^r$ ($r \in \bbN$). 

Moreover, $\cB(\cH)$ denotes the class of bounded linear operators, which map the Hillbert space $\cH$ into itself. Bounded operators which map the Hilbert space $\cH_1$ into the Hilbert 
space $\cH_2$ are denoted by $\cB(\cH_1, \cH_2)$.

By $C^k(\Om)$ we denote the usual class of functions on the open set $\Om \subseteq \bbR^n$ for some $n \in \bbN$, whose partial derivatives up to order $k$ are continuous, and $C(\Om)$ stands for continuous functions on $\Om$. Similarly, $C^k(\ol{\Om})$ denotes the space of functions whose partial derivatives up to order $k$ are bounded and uniformly continuous on $\Om$ (see, e.g., 
\cite[Ch.~1]{AF03}). When functions belong to $C^k(\Om)$, we say that they are $k$ times continuously differentiable. 

We employ the notation $\bbC_+ = \{z \in \bbC \, | \, \Im(z) > 0\}$ for the open complex upper half-plane.

%%%%%%%%%%
%%%%%%%%%%
\section{Preliminaries} \lb{s2}
%%%%%%%%%%
%%%%%%%%%%

Fundamental solutions of \eqref{A2} play an essential role in this theory
(it is always assumed that the entries of $v$ are locally integrable on $[0,\infty)$ and so the fundamental solutions exist).
See Notation \ref{NnFS} and formula \eqref{e1} for the  normalizations of the considered fundamental solutions.
In view of \eqref{e1} one has the equality
\begin{align}\lb{e3}&
u_{0}(x,z)=u_{\ell}(x,z)u_{0}(\ell,z).
\end{align}
Taking into account Definition \ref{DnWF} and formula \eqref{e3} one derives 
\begin{align}\lb{e4}&
u_{\ell}(\dott ,z)u_{0}(\ell,z)\begin{bmatrix}I_{m_1} \\ \vp_{0}(z)\end{bmatrix} 
\in \big[L^2((0, \infty))\big]^{m \times m_1}.
\end{align}
According to the representation \cite[eq.~(2.26)]{SSR13} of $\vp_{0}(z)$ (see \cite[Cor.~2.21]{SSR13})
and by the corresponding formula (2.29) in \cite{SSR13}, one obtains 
\begin{align}\lb{e5}&
\begin{bmatrix}I_{m_1} & \vp_{0}(z)^*\end{bmatrix} u_{0}(\ell,z)^* J u_{0}(\ell,z)\begin{bmatrix}I_{m_1} \\ \vp_{0}(z)\end{bmatrix} \geq 0_{m_1}.
\end{align}
Next, similarly to \eqref{A3}, we now partition  $u_{0}(\ell,z)$ into four blocks: $u_{0}(\ell,z)=\{u_{0,jk}(\ell,z)\}_{j,k=1}^2$. Then formula \eqref{e5} yields
the fact
\begin{align}\lb{e6}
\det\big(u_{0,11}(\ell,z)+u_{0,12}(\ell,z)\vp_0(z)\big) \neq 0.
\end{align}
Therefore, one obtains 
\begin{align}& \lb{e7}
u_{0}(\ell,z)\begin{bmatrix}I_{m_1} \\ \vp_{0}(z)\end{bmatrix} =\begin{bmatrix}I_{m_1} \\ \wt \vp_{\ell}(z)\end{bmatrix} \big(u_{0,11}(\ell,z)+u_{0,12}(\ell,z)\vp_0(z)\big),
\\  & \lb{e8} \wt \vp_{\ell}(z):=\big(u_{0,21}(\ell,z)+u_{0,22}(\ell,z)\vp_0(z)\big)\big(u_{0,11}(\ell,z)
+u_{0,12}(\ell,z)\vp_0(z)\big)^{-1}.
\end{align}
Formulas \eqref{e4}, \eqref{e6}, and \eqref{e7} imply that
\begin{align}\lb{e9}&
u_{\ell}(\dott ,z)\begin{bmatrix}I_{m_1} \\ \wt \vp_{\ell}(z)\end{bmatrix} \in \big[L^2((\ell, \infty))\big]^{m \times m_1}.
\end{align}
Comparing \eqref{e2} and \eqref{e9}, and taking into account that the Weyl--Titchmarsh function is unique, one infers that $\vp_{\ell}(z)=\wt \vp_{\ell}(z)$,
and hence \eqref{e8} takes on the form
\begin{align}& \lb{e10}
 \vp_{\ell}(z)=\big(u_{0,21}(\ell,z)+u_{0,22}(\ell,z)\vp_0(z)\big)\big(u_{0,11}(\ell,z)
 +u_{0,12}(\ell,z)\vp_0(z)\big)^{-1}.
\end{align}

%%%%%
\begin{theorem} Suppose that $v \in \big[L^1((0,R))\big]^{m_1 \times m_2}$ for all $R > 0$ and assume that  $u_{0}(\ell,z)=\{u_{0,jk}(\ell,z)\}_{j,k=1}^2$ represents the value of the fundamental solution $u_0$ of the system
\eqref{A2} at the point $x=\ell$. Then the Weyl--Titchmarsh functions $ \vp_{\ell}$ and $ \vp_{0}$ of the system \eqref{A2} on the semi-axes $[\ell, \infty)$
and $[0, \infty)$, respectively, are connected by relation \eqref{e10}.
\end{theorem}
%%%%%

%%%%%
\begin{proposition} \lb{PnHCtr}
Let some $m \times m$ matrix function $w(x)$ satisfy the differential equation $w^{\prime}(x)=G(x)w(x)$, where $G \in \big[L^1((0,R))\big]^{m \times m}$ for all $R > 0$ has the block form
$G=\{G_{jk}\}_{j,k=1,2}^2$ and $G_{jj}$ is an $m_j \times m_j$ matrix-valued function $(j=1,2)$.  Let $P$ 
and $Q$ be $m_1 \times m_1$ and $m_2 \times m_1$
matrices, respectively, and suppose that $\det(G_{11}(x)P+G_{12}(x)Q)\neq 0$ holds on some open subset 
$\Omega \subseteq \BR$. Then $($for a.e.~$x\in \Omega$$)$, the linear fractional transformation 
\begin{equation}
\phi(x)=(G_{21}(x)P+G_{22}(x)Q)(G_{11}(x)P+G_{12}(x)Q)^{-1} 
\end{equation} 
satisfies the following matrix-valued Riccati-type differential equation $($for a.e.~$x\in \Omega$$)$
\begin{align}& \lb{e11}
\phi^{\prime}(x)=-\phi(x)G_{12}(x)\phi(x)-\phi(x)G_{11}(x)+G_{22}(x)\phi(x)+G_{21}(x).
\end{align}
\end{proposition}
%%%%%%

In view of \eqref{e2} and \eqref{e10}, we rewrite \eqref{e11}, for the case of the Weyl--Titchmarsh function $\vp_{\ell}(z)$, in the form, 
\begin{align}& \lb{e12}
\frac{d}{d \ell}\vp_{\ell}(z)=- i \big(\vp_{\ell}(z)v(\ell)\vp_{\ell}(z)+v(\ell)^*+2z\vp_{\ell}(z)\big).
\end{align}

%%%%%%%%%
\section{The $\cla$-Function for Dirac-Type Systems: Part I} \lb{s3}
%%%%%%%%%

Analogs of the $\cla$-function for Dirac-type systems were considered in \cite{Sa02} 
 for the case $m_1=m_2$ (see also some related results in \cite{Sa88}). The case of rectangular potentials $v$ was dealt with in \cite{FKRS12, SSR13, Sa15}. 

In order to construct the $\cla$-equation for Dirac-type system we need an analog of the $A$-function representation
\eqref{A1+} of the Weyl--Titchmarsh function $\vp$. 
The results were developed further in \cite{Sa02} and for the rectangular $m_2 \times m_1$ Weyl--Titchmarsh functions
in \cite{FKRS12, Sa15}.

In this paragraph, we  present some  results from \cite{Sa15} (more precisely, Proposition 3.1, Theorem 4.1, and  Corollary 4.2 in
\cite{Sa15}) that we need for our considerations.  We recall that $u_0(x,z)$ is the fundamental solution of
system \eqref{A2} and now partition $u_0(x,0)$ into block rows
\begin{align}& \lb{e13}
\b(x)=\begin{bmatrix}\b_1(x) & \b_2(x)\end{bmatrix} 
:= \begin{bmatrix}I_{m_1} & 0_{m_1 \times m_2}\end{bmatrix} u_0(x,0), 
\\ & \lb{e14}
\g(x)=\begin{bmatrix}\g_1(x) & \g_2(x)\end{bmatrix}
:=\begin{bmatrix}0_{m_2 \times m_1} & I_{m_2} \end{bmatrix} u_0(x,0),
\end{align}
where $\b_1$ and $\g_2$ are $m_1\times m_1$ and $m_2\times m_2$ matrix-valued functions, respectively. In view of \eqref{A2} one obtains
\begin{align}& \lb{e14+}
 \g^{\prime}(x) =-i v(x)^* \b(x).
\end{align}
The system \eqref{A2} also yields $u_0(x,0)^* J u_0(x,0)= J =u_0(x,0) J u_0(x,0)^*$. In particular, one has
\begin{align}& \lb{e15-}
 \g(x) J \g(x)^* \equiv -I_{m_2}, \quad  \b(x) J \g(x)^* \equiv 0_{m_1 \times m_2}.
\end{align}
From \eqref{e14+} and the second equality in \eqref{e15-} it follows that
\begin{align}& \lb{e15'}
 \g^{\prime}(x) J \g(x)^* \equiv 0_{m_2}.
 \end{align}
The following operators $K$, acting in $\big[L^2((0,\eT))\big]^{m_2}$ ($\eT \in (0,\infty)$),  
\begin{align}& \lb{e15}
(K f)(x)=i \g(x) \int_0^x J \g(t)^* f(t) \, dt, \quad f \in \big[L^2((0,\eT))\big]^{m_2}, 
\end{align}
play an essential role in solving the inverse problem of recovering $v$ from the Weyl--Titchmarsh function $\varphi$.

%%%%%% 
\begin{proposition} \lb{PnSimN} Suppose that $v \in \big[L^2((0,\eT))\big]^{m_1 \times m_2}$ 
$($$\eT \in (0,\infty)$$)$ in the Dirac-type  system \eqref{A2}, and let $K$ be given by  \eqref{e15}, where $\g$ is defined in \eqref{e14}.
Then there is a similarity transformation operator $E \in \cB\Big(\big[L^2((0, \eT))\big]^{m_2}\Big)$ such that 
\begin{align}
& K=E \cI E^{-1}, \quad  (\cI f)(x):=-i \int_0^x f(t) \,dt \quad (x > 0),   \lb{i1} \\
\begin{split} 
& (E f)(x)= f(x) + \int_0^x N(x,t) f(t) \,dt,  \quad f \in \big[L^2((0,\eT))\big]^{m_2}    \\
&\hspace*{5.8cm} (\text{for a.e.~$x \in (0, \eT)$}),   \lb{i2} 
\end{split} \\
&(E^{-1}\g_2)(x)\equiv I_{m_2} \quad (x > 0),     \lb{i3} 
\end{align}
where 
 $N(\dott, \dott)$ is a Hilbert--Schmidt integral kernel, $E^{-1}$ is applied to $\g_2(x)$ in \eqref{i3} columnwise,
 and \eqref{i3} has to be understood in the following way: the image of the $k$-th column of $\g_2(x)$ equals 
 the constant vector function given by the $k$-th column of $I_{m_2}$ $(1\leq k\leq m_2)$.
Moreover, the operators $E^{\pm 1}$ map differentiable vector-valued functions with a square integrable derivative into vector-valued functions of the same class.
\end{proposition}
%%%%%%

%%%%%%
\begin{remark}\lb{RkIm} Formulas similar to \eqref{i3} will also appear in the remainder of this paper and we understand them in the analogous manner: namely, operators are applied to matrix functions columnwise, and $I_{m_2}$ in this context is understood as
the corresponding constant matrix function, which takes the values $I_{m_2}$ at each $x$.
${}$ \hfill $\diamond$
\end{remark}

%%%%%%
\begin{remark}\lb{Rktilde}
The operator $E$ satisfying the conditions of Proposition~\ref{PnSimN} is constructed in the proof
of \cite[Proposition~3.1]{Sa15} as a product $E=\wt E E_0$, where $\wt E$ satisfies all the conditions
of  Proposition \ref{PnSimN} excluding, possibly, \eqref{i3} and
\begin{align}      
\begin{split} 
& (E_0 f)(x)= f(x) + \int_0^x N_0(x-t) f(t) \, dt, \quad f \in \big[L^2((0,\eT))\big]^{m_2},   \lb{i7} \\
& N_0(x):=\big(\wt E^{-1}\g_2\big)^{\prime}(x) \quad (\text{for a.e.~$x \in (0, \eT)$}). 
\end{split} 
\end{align} 
More precisely, there is a whole family of triangular operators satisfying all the conditions
of  Proposition \ref{PnSimN} excluding, possibly, \eqref{i3}, and $\wt E$ is one of them.
The construction of the similarity transformation operator $\wt E$ is based on the work \cite{LAS}
and uses the resolvent of $K$. We normalize $\wt E$ using the triangular convolution operator
$E_0$ so that \eqref{i3} holds for $E=\wt E E_0$. (One notes that $E_0$ depends on $\wt E$, see the definition of $N_0$ in \eqref{i7}.)
${}$ \hfill  $\diamond$
\end{remark}
%%%%%%

Using the operator $E$ we now introduce the $m_2 \times m_1$ matrix-valued function 
\begin{align}& \lb{e16}
\Phi(x):=\big(E^{-1}\g_1\big)(x) \quad (0<x<\eT).
\end{align}
According to Proposition \ref{PnSimN}, $\Phi$ is differentiable and, in the case of locally
bounded potentials $v$ (see \cite[eq.~(2.154)]{SSR13}) the following equality holds:
\begin{align}& \lb{e16!}
v(x)=\big(i E\Phi^{\prime}\big)(x)^*  \quad (0<x<\eT).
\end{align} 
%%%%%%%

%%%%%%%  
\begin{remark}\lb{Rktilde2}
We note that, by construction,  the function $N(x,t)$ in \eqref{i2} does not depend (for any fixed $\eT_0>0$ and values $x < \eT_0$) on 
the choice  of $\eT \geq  \eT_0$. This means that \eqref{e16} (for various values of $\eT$) uniquely determines $\Phi(x)$
on the whole half-axis $(0, \infty)$.
Taking into account \cite[Proposition 1.3]{Sa15}, we now reformulate Theorem 4.1  and  Corollary 4.2 in \cite{Sa15} in the following manner.  \hfill  $\diamond$
\end{remark}
%%%%%%

%%%%%%
\begin{proposition}\lb{cyHea} Let $\vp \equiv \vp_0$ be the Weyl--Titchmarsh function of the Dirac  system  \eqref{A2} on  $[0, \infty)$, assuming $v \in \big[L^2((0,R))\big]^{m_1 \times m_2}$ for all $R>0$. Then 
$($uniformly with respect to $\Re(z))$ one has 
\begin{align} \lb{i27}&
\vp(z) \underset{\Im (z) \to \infty}{=} 2i z\int_0^{\eT}e^{2i xz}\Phi(x) \, dx 
+ O\Big(z e^{2i \eT z} \big/ (\Im(z))^{1/2}\Big),
\end{align} 
implying 
\begin{align} \lb{Repr}&
\vp(z)=2i z\int_0^{\infty}e^{2i xz}\Phi(x) \, dx \quad (\Im (z) >0).
\end{align} 
\end{proposition}
%%%%%%

%%%%%%
\begin{remark} \lb{r3.6} 
Formula \eqref{i27} is an analog of equality \eqref{A1+} and the $\cla$-equation for Dirac  systems will be formulated
in terms of $\Phi$ and its derivatives. In order to use \eqref{i27} or \eqref{Repr} and rewrite \eqref{e12} in terms of
$\Phi$ as an $\cla$-equation, we should integrate the integrals in \eqref{i27} or \eqref{Repr} by parts.
Therefore, $\Phi(x)$ should be two times differentiable. \hfill  $\diamond$
\end{remark}
%%%%%%

%%%%%%
\section{The $\cla$-Function for Dirac-Type Systems: Part II} \lb{s4}
%%%%%%

 In this section we consider the case 
$v\in \big[C^1([0, \infty))\big]^{m_1 \times m_2}$. Our next proposition deals with the differentiability 
of $\Phi$ mentioned at the end of Remark \ref{r3.6}.

%%%%%%
\begin{proposition}\lb{PnDif} Assume that  $v \in \big[C^1([0, \infty))\big]^{m_1 \times m_2}$ in the Dirac-type  system \eqref{A2}. Then $($cf.\ \eqref{e16}$)$
\begin{equation} 
\Phi(x) \in \big[C^2([0, \infty))\big]^{m_2 \times m_1}.
\end{equation} 
\end{proposition}
%%%%%%
\begin{proof} We divide the proof into four steps. \\[1mm]  
{\bf Step 1.}  We consider the operator $\wt E$ in $\big[L^2((0,\eT))\big]^{m_2}$ given by
\begin{equation} 
\big(\wt E f\big)(x)=  f(x) + \int_0^x \wt N(x,t) f(t) \,dt, \quad f \in \big[L^2((0,\eT))\big]^{m_2}, 
\lb{e17-} 
\end{equation}
(for a.e.~$x \in (0,\eT)$) discussed in Remark \ref{Rktilde}. In view of Proposition \ref{PnSimN} 
and Remark \ref{Rktilde} one has 
\begin{align}
& K=\wt E \cI \wt E^{-1} .    \lb{e17'}
\end{align}
Next, we show that $\wt E^{-1}$ maps vector-valued $C^2([0, \eT])^{m_2}$ functions 
into vector-valued functions of the same class.

Assuming $v \in \big[C^1([0, \infty))\big]^{m_1 \times m_2}$, the construction of $\wt N(\dott,\dott)$ in 
\cite[Lemma~2.3]{Sa15} then implies that $\wt N(x,t)$ in \eqref{e17-} is continuous on the domain 
$0~\leq~t~\leq~x~ < ~\infty$, and that
\begin{align} \lb{e17}&
\wt N(x,0) = 0_{m_2}.
\end{align} 
Since $\wt N(x,t)$ is continuous, the integral kernel $\wt N^{\times}(x,t)$ of the operator
\begin{align} \lb{e18}&
\big( \wt E^{-1} f\big)(x) = f(x) + \int_0^x \wt N^{\times}(x,t) f(t) \,dt, \quad f \in \big[L^2((0,\eT))\big]^{m_2}  
\end{align} 
(for $x \in (0,\eT)$), is continuous as well.  Moreover, using \eqref{e17-} and \eqref{e18} in order to 
detail the integral kernel of $\wt E^{-1} \wt E$, one obtains
\begin{align} \lb{e19}&
\wt N^{\times}(x,t)+\wt N(x,t)+\int_t^x \wt N^{\times}(x,r) \wt N(r,t) \, dr=0_{m_2}.
\end{align} 
Clearly, \eqref{e17} and \eqref{e19} yield
\begin{align} \lb{e20}&
\wt N^{\times}(x,0) = 0_{m_2}.
\end{align} 

We shall need \eqref{e19} and \eqref{e20} in future considerations, but first we consider
$K$ in greater detail. Clearly, $v \in \big[C^1([0, \infty))\big]^{m_1 \times m_2}$ implies that 
$\g \in \big[C^2([0, \infty))\big]^{m_2 \times m}$. It is immediate from \eqref{e15} that (for $x \in (0,\eT)$)
\begin{align}
\begin{split} 
& \big(Kf\big)^{\prime}(x)=i \g(x) J \g(x)^* f(x)+i \g^{\prime}(x)\int_0^x J \g(t)^*f(t) \, dt,   \lb{e21} \\
& \big(Kf\big)(0)=0_{m_2 \times 1}, \quad f \in \big[L^2((0,\eT))\big]^{m_2}.
\end{split} 
\end{align} 
Taking into account \eqref{e15-} and \eqref{e21}, one rewrites $K$ in the form (for $x \in (0,\eT)$)
\begin{align} 
\begin{split} 
& (K f)(x) = i \cI (Kf)^{\prime}(x) = (\cI (I_{[L^2((0,\eT))]^{m_2}} - K_1)f)(x),   \lb{e22} \\ 
& (K_1 f)(x):= \g^{\prime}(x)\int_0^x J \g(t)^* f(t) \, dt, \quad f \in \big[L^2((0,\eT))\big]^{m_2}.
\end{split} 
\end{align} 
In particular, $K = \cI (I_{[L^2((0,\eT))]^{m_2}} - K_1)$. 
Rewriting \eqref{e17'} and using \eqref{e22}, one arrives at 
\begin{align} \lb{e23}&
\wt E^{-1}K= \cI \wt E^{-1}, \quad \wt E^{-1} \cI = \cI \wt E^{-1}(I_{[L^2((0,\eT))]^{m_2}} - K_1)^{-1}.
\end{align} 
{\bf Step 2.} We denote the integral kernels of the operators $(I_{[L^2((0,\eT))]^{m_2}} - K_1)$ and 
$(I_{[L^2((0,\eT))]^{m_2}} - K_1)^{-1}$ by $-N_1$ and by $-N_1^{\times}$, respectively. That is, we set (for $x \in (0,\eT)$) 
\begin{align} 
& (K_1 f)(x)= \int_0^xN_1(x,t) f(t) \, dt, \quad N_1(x,t)= \g^{\prime}(x) J \g(t)^*,   \lb{e24} \\  
& \big((I_{[L^2((0,\eT))]^{m_2}} - K_1)^{-1} f\big)(x) = f(x) - \int_0^xN_1^{\times}(x,t) f(t) \, dt, \quad 
f \in \big[L^2((0,\eT))\big]^{m_2}.   \lb{e25}
\end{align} 
One recalls that $\g \in \big[C^2([0, \infty))\big]^{m_2 \times m}$, and so $N_1(\dott,\dott)$ (given in \eqref{e24})
is differentiable with a continuous derivative. By virtue of \eqref{e15'} and \eqref{e24}
one obtains $N_1(x,x)=0_{m_2}$.

Since $N_1(\dott,\dott)$ is continuous, $N_1^{\times}(\dott, \dott)$ is continuous as well. 
In order to show that   $N_1^{\times}(\dott,\dott)$ is differentiable, we consider  
the integral kernel of $(I_{[L^2((0,\eT))]^{m_2}} - K_1)(I_{[L^2((0,\eT))]^{m_2}} - K_1)^{-1}$, and (similar to \eqref{e19}) derive
\begin{align} \lb{e26}&
N_1^{\times}(x,t)+ N_1(x,t)=\int_t^x   N_1(x,r)  N_1^{\times}(r,t) \, dr.
\end{align} 
Taking into account  that the derivatives of $N_1$ are continuous and using 
\eqref{e26}, one concludes that $\frac{\p }{\p x} N_1^{\times}(x,t)$ is continuous. 
In view of the relations \eqref{e26} and $N_1(x,x)=0_{m_2}$, one gets $N_1^{\times}(x,x)=0_{m_2}$.
Hence, we next introduce the operator $K_2$ in $\big[L^2((0,\eT))\big]^{m_2}$ by the equalities (for $x \in (0,\eT)$) 
\begin{align}
\begin{split}
\big(K_2f\big)(x)=\frac{d}{d x}\int_0^xN_1^{\times}(x,t) f(t) \, dt
= \int_0^x \left(\frac{\p }{\p x}N_1^{\times}\right)(x,t) f(t) \, dt,&    \lb{e27} \\
f \in \big[L^2((0,\eT))\big]^{m_2}.&
\end{split} 
\end{align} 
Relations 
\eqref{e25} and \eqref{e27} then yield the following representation of $(I_{[L^2((0,\eT))]^{m_2}} - K_1)^{-1}$:
\begin{align} 
& (I_{[L^2((0,\eT))]^{m_2}} - K_1)^{-1} = I_{[L^2((0,\eT))]^{m_2}} - i \cI  K_2,    \no \\ 
& (K_2 f)(x) = \int_0^xN_2(x,t) f(t) \, dt,  \quad f \in \big[L^2((0,\eT))\big]^{m_2},     \lb{e28} \\
& N_2(x,t) = \bigg(\frac{\p }{\p x}N_1^{\times}\bigg)(x,t), \quad (x,t) \in [0,\eT] \times [0,\eT],   \no 
\end{align} 
where $N_2(\dott, \dott)$ is continuous on $[0,\eT] \times [0,\eT]$. From \eqref{e23} and \eqref{e28} one infers that 
\begin{equation} 
\wt E^{-1} \cI = \cI \wt E^{-1}-i \cI \wt E^{-1} \cI K_2,   \lb{2.42}
\end{equation} 
and the substitution of \eqref{e23} into the right-hand side of this equality results in 
\begin{align} \lb{e29}&
 \wt E^{-1} \cI = \cI\wt E^{-1}-i \cI^2\wt E^{-1}(I_{[L^2((0,\eT))]^{m_2}} - K_1)^{-1}K_2.
\end{align} 
Right-multiplying both sides of \eqref{e29} by $\cI$ one gets 
\begin{align} \lb{e29'}&
\wt E^{-1} \cI^2 = \cI \wt E^{-1} \cI - i \cI^2\wt E^{-1}(I_{[L^2((0,\eT))]^{m_2}} - K_1)^{-1}K_2 \cI.
\end{align} 
Substituting   \eqref{e29} into the right-hand side of \eqref{e29'} one derives
\begin{align} \nn 
 \wt E^{-1} \cI^2 & = \cI(\cI \wt E^{-1}-i \cI^2 \wt E^{-1}(I_{[L^2((0,\eT))]^{m_2}} - K_1)^{-1}K_2)  \no \\
& \quad  - i \cI^2\wt E^{-1}(I_{[L^2((0,\eT))]^{m_2}} - K_1)^{-1}K_2 \cI     \no \\
&= \cI^2 \big(\wt E^{-1}-i \cI \wt E^{-1}(I_{[L^2((0,\eT))]^{m_2}} - K_1)^{-1}K_2    \no \\ 
& \quad - i \wt E^{-1}(I_{[L^2((0,\eT))]^{m_2}} 
- K_1)^{-1}K_2 \cI\big).   \lb{e30}
\end{align} 
Finally, we note that any $f\in \big[C^2([0, \eT])\big]^{m_2}$ admits the representation
\begin{align} \lb{e31}&
f(x)=f(0)+i \cI f^{\prime}(0)- \big(\cI^2f^{\prime \prime}\big)(x).
\end{align} 
Recalling that $\cI$ is the integration operator (multiplied by $- i$) and taking into account
\eqref{e29}--\eqref{e31}  one concludes that $\wt E^{-1}f\in \big[C^2([0, \eT])\big]^{m_2}$ for 
any $f\in \big[C^2([0, \eT])\big]^{m_2}$ if only  $\big(\wt E^{-1}f(0)\big)(x)$ belongs $\big[C^2([0, \eT])\big]^{m_2}$
for any constant vector function $f(0)$. Recalling Remark \ref{RkIm} we rewrite this condition as
\begin{align} \lb{e32}&
\big(\wt E^{-1}I_{m_2}\big)(x) \in \big[C^2([0, \eT])\big]^{m_2 \times m_2}.  
\end{align} 
{\bf Step 3.} We rewrite the equality $\wt E^{-1}K = \cI \wt E^{-1}$ in \eqref{e23} in 
terms of the integral kernels of the corresponding operators:
\begin{align} \lb{e33}&
I_{m_2}+\g(x) J \g(t)^*+ \int_t^x \wt N^{\times}(r,t) \, dr 
+ \int_t^x \wt N^{\times}(x,r)\g(r) \, dr \, J \g(t)^*=0_{m_2}.
\end{align} 
Setting $t=0$ in \eqref{e33} and recalling \eqref{e20}, one obtains
\begin{align} \lb{e34}&
\wt E^{-1}\big(- \g(x) J \g(0)^*\big)= I_{m_2}.
\end{align} 
Formulas  \eqref{e15-}, \eqref{e15'} and \eqref{e31} imply that
\begin{align} \lb{e35}&
 -\g(x) J \g(0)^*= I_{m_2} + \cI^2 \g^{\prime \prime}(x) J \g(0)^*.
\end{align} 
It is immediate from \eqref{e30} that
\begin{align} \lb{e36}&
\wt E^{-1} \cI^2 \g^{\prime \prime}(x) J \g(0)^* \in \big[C^2([0, \eT])\big]^{m_2 \times m_2},
\end{align} 
and \eqref{e32} follows from \eqref{e34}--\eqref{e36}.
This completes the proof that $\wt E^{-1}f\in \big[C^2([0, \eT])\big]^{m_2}$ for any 
$f\in \big[C^2([0, \eT])\big]^{m_2}$.

{\bf Step 4.} One recalls that relations \eqref{e16} (varying $\eT \in (0,\infty)$) uniquely determine
$\Phi$ on $[0, \infty)$ and that according to Remark \ref{Rktilde} one has 
$E=\wt E E_0$
for $E$ in \eqref{e16}. Thus, in order to prove Proposition \ref{PnDif} it remains to show that
\begin{align} \lb{e37}&
E_0^{-1}f \in \big[C^2([0, \eT])\big]^{m_2} \, \text{ for any } \, f \in \big[C^2([0, \eT ])\big]^{m_2},
\end{align} 
where $E_0$ has the form \eqref{i7}.
It is easy to see (and is used in the proof
of \cite[Proposition 3.1]{Sa15}) that
\begin{align} \lb{e38}&
E_0^{-1} \cI = \cI E_0^{-1}.
\end{align} 
In view of \eqref{e31} and \eqref{e38}, we (similarly to the case of $\wt E^{-1}$) see that
$\big(E_0^{-1}f\big)(x)\in \big[C^2([0, \eT])\big]^{m_2}$ if $E_0^{-1}$ applied to the constant vector $f(0)$ belongs $\big[C^2([0, \eT])\big]^{m_2}$.
In other words, it remains to show that $E_0^{-1}I_{m_2}$ belongs $\big[C^2([0, \eT])\big]^{m_2\times m_2}$.

In order to prove $\big(E_0^{-1}I_{m_2}\big)(x) \in \big[C^2([0, \eT])\big]^{m_2 \times m_2}$, one notes that
according to \eqref{i7} one has 
\begin{align} 
\begin{split} 
\big(E_0 I_{m_2}\big)(x)& = I_{m_2}+\int_0^x N_0(t) \, dt=I_{m_2}+i\cI \big(\big(\wt E^{-1}\g_2\big)^{\prime}(x)\big)
\\  &
=I_{m_2}+i \cI \Big(\big(\wt E^{-1}\g_2\big)^{\prime}(0) 
+ i \cI \big(\wt E^{-1}\g_2\big)^{\prime\prime}(x)\Big),
\end{split} 
\end{align} 
which (using \eqref{e38}) may be rewritten in the form
\begin{align} \lb{e39}&
\big(E_0^{-1}I_{m_2}\big)(x) = I_{m_2}-i \cI E_0^{-1}\Big(\big(\wt E^{-1}\g_2\big)^{\prime}(0)+i \cI \big(\wt E^{-1}\g_2\big)^{\prime\prime}(x)\Big).
\end{align} 
Clearly, the right-hand side of \eqref{e39} belongs to $\big[C^1([0,\eT])\big]^{m_2 \times m_2}$, and so \\
$\big(E_0^{-1}I_{m_2}\big)(x)$ belongs to $\big[C^1([0,\eT])\big]^{m_2 \times m_2}$. Hence, 
$E_0^{-1}\big(\wt E^{-1}\g_2\big)^{\prime}(0)$ belongs to $\big[C^1([0,\eT])\big]^{m_2 \times m_2}$, 
which implies that the right-hand side of \eqref{e39} belongs to $\big[C^2([0,\eT])\big]^{m_2 \times m_2}$,
and the required relation $\big(E_0^{-1}I_{m_2}\big)(x) \in \big[C^2([0, \eT])\big]^{m_2 \times m_2}$ follows.
\end{proof}
%%%%%%

%%%%%%
\begin{remark}\lb{RkSdvig}
Together with the system \eqref{A2} one can consider the shifted systems
\begin{align} &       \lb{e40}
\wh y_{\ell}^{\prime}(x, z )=i \big(z J + J \wh V_{\ell}(x)\big)\wh y_{\ell}(x,z), \quad \wh V_{\ell}(x):=V(x+\ell) 
\quad (x \geq 0), 
\end{align} 
where $\ell \geq 0$ and the fundamental solution is denoted by $\wh u_{\ell}(x,z)$. It is easy to see that
 \begin{align} &       \lb{e41}
  \wh u_{\ell}(x,z)=u_{\ell}(x+\ell,z).
\end{align} 
Equality  \eqref{e41} and Definition \ref{DnWF} imply that $\vp_{\ell}(z)$ is the Weyl--Titchmarsh function of
the system \eqref{e40}. \hfill  $\diamond$
\end{remark}
%%%%%

%%%%%
\begin{notation}  Henceforth, we introduce the additional parameter $\ell$ in our notation of the corresponding
systems \eqref{e40} $($instead of the system \eqref{A2}$)$ and write, for instance,  $\b(x,\ell)$, $\g(x,\ell)$,
$\Phi(x,\ell)$, etc. In addition, we set 
 \begin{align} &       \lb{e42}
\cla(x, \ell):=\frac{\p}{\p x}\Phi(x,\ell) \equiv \Phi^{\prime}(x,\ell).
\end{align} 
\end{notation}

Clearly, \eqref{e16!} holds pointwise in the case $v\in \big[C^1([0, \infty))\big]^{m_1 \times m_2}$. Formulas  \eqref{i2} and \eqref{e16!}, together with the   equalities 
$\wh V_{\ell}(x):=V(x+\ell)$ and \eqref{e42}  above, imply that
 \begin{align} &       \lb{e43}
v(0)= -i   \Phi^{\prime}(0)^*, \quad     v( \ell)=-i \Phi^{\prime}(0,\ell)^*=-i \cla (0,\ell)^*.
\end{align} 

%%%%%%
\section{The $\cla$-Equation for Dirac-Type Systems} \lb{s5}
%%%%%%

Using \eqref{i27} and \eqref{e43} we now rewrite \eqref{e12} in the form of the $\cla$-equation for
Dirac-type  systems and obtain the following result.

%%%%%%
\begin{theorem}\lb{AeqfD} 
Suppose that \, $v \in \big[C^1([0, \infty))\big]^{m_1 \times m_2}$ in the Dirac-type system \eqref{A2}. Then the 
$\cla(\dott,\dott)$-function given by
\eqref{e42} is jointly continuously differentiable in both variables, 
$\cla(\dott,\dott) \in \big[C^1([0,\infty) \times [0,\infty))\big]^{m_2 \times m_1}$, and satisfies the integro-differential equation $($for $(x,\ell) \in [0,\infty) \times [0,\infty)$$)$
 \begin{align} &       \lb{e44}
\frac{\p}{\p \ell}\cla(x, \ell) = \frac{\p}{\p x}\cla(x,\ell)
+ \int_0^x \cla(x-t,\ell)\cla(0,\ell)^*\cla(t,\ell) \, dt.
\end{align} 
\end{theorem}
%%%%%%
\begin{proof} We divide the proof into two steps. \\[1mm] 
{\bf Step 1.} Formula \eqref{e16} implies that
 \begin{align} &       \lb{e45}
\Phi(0)=\big(E^{-1}\g_1\big)(0)=\g_1(0)=0_{m_2 \times m_1}.
\end{align} 
In view of \eqref{i27}, \eqref{e42} and \eqref{e45} one infers that 
 \begin{align} 
 \begin{split} 
\vp_{\ell}(z)& \underset{\Im (z) \to \infty}{=} \int_0^{\eT}\Phi(x,\ell) \, d\big(e^{2 i x z}\big) 
+ O\Big(z e^{2 i \eT z}\big/ (\Im(z))^{1/2}\Big)   \\       
& \underset{\Im (z) \to \infty}{=} - \int_0^{\eT}e^{2i x z}\cla(x,\ell) \, dx 
+ O\Big(z e^{2 i \eT z}\big/ (\Im(z))^{1/2}\Big).   \lb{e46}
\end{split} 
\end{align} 
According to Proposition  \ref{PnDif}, $\cla(x, \ell)$ is  continuously differentiable with respect to  $x$.
Hence, taking into account \eqref{e43}, one rewrites \eqref{e40} in the form
 \begin{align} 
 \begin{split} 
2z\vp_{\ell}(z)& \underset{\Im (z) \to \infty}{=} - i \cla(0,\ell)- i \int_0^{\eT}e^{2i x z}\cla^{\prime}(x,\ell) \, dx 
+O\Big(z^2 e^{2 i \eT z}\big/ (\Im(z))^{1/2}\Big)      \\       
& \underset{\Im (z) \to \infty}{=} - v(\ell)^*- i \int_0^{\eT}e^{2i x z}\cla^{\prime}(x,\ell) \, dx 
+ O\Big(z^2 e^{2 i \eT z}\big/ (\Im(z))^{1/2}\Big).   \lb{e47}
\end{split} 
\end{align}
Moreover, using \eqref{e46} and \eqref{e43}, the representation of $\vp_{\ell}(z)v(\ell)\vp_{\ell}(z)$ reads as follows 
\begin{align} 
\vp_{\ell}(z)v(\ell)\vp_{\ell}(z) & \underset{\Im (z) \to \infty}{=} - i \int_0^{\eT} 
\int_0^{\eT}e^{2i (\tau+t)z}\cla(\tau, \ell) \, d\tau\cla(0,\ell)^*\cla(t,\ell) \, dt    \no \\  
& \hspace*{1.4cm} + O\Big(z e^{2 i \eT z}\big/ (\Im(z))^{1/2}\Big)    \no \\
\begin{split} 
& \underset{\Im (z) \to \infty}{=} - i \int_0^{\eT} \int_t^{\eT+t}e^{2i x z}\cla(x-t, \ell) \, dx\cla(0,\ell)^*\cla(t,\ell) \, dt
\\  
& \hspace*{1.4cm} + O\Big(z e^{2 i \eT z}\big/ (\Im(z))^{1/2}\Big).    \lb{e48} 
\end{split} 
\end{align}
Changing the order of integration  in \eqref{e48} (and removing the swiftly decaying part in the result) one arrives at 
\begin{align} 
\begin{split}
\vp_{\ell}(z)v(\ell)\vp_{\ell}(z) & \underset{\Im (z) \to \infty}{=}
-i \int_0^{\eT} e^{2i x z}\int_0^x\cla(x-t, \ell)\cla(0,\ell)^*\cla(t,\ell) \, dt \, dx    \\
& \hspace*{1.4cm} + O\Big(z e^{2 i \eT z}\big/ (\Im(z))^{1/2}\Big).      \lb{e49}
\end{split} 
\end{align}
Taking into account \eqref{e47} and \eqref{e49} one rewrites \eqref{e12} in the form
\begin{align} 
\begin{split} 
\frac{d}{d \ell}\vp_{\ell}(z)& \underset{\Im (z) \to \infty}{=} - \int_0^{\eT} e^{2i x z} 
\int_0^x\cla(x-t, \ell)\cla(0,\ell)^*\cla(t,\ell) \, dt \, dx     \\
& \hspace*{1.4cm}
-\int_0^{\eT}e^{2i x z}\cla^{\prime}(x,\ell) \, dx 
+ O\Big(z^2 e^{2 i \eT z}\big/ (\Im(z))^{1/2}\Big).   \lb{e50}
\end{split} 
\end{align}

{\bf Step 2.}  A considerable part of the proof connected with the representation of 
$\frac{d}{d \ell}\vp_{\ell}(z)$ is moved to Appendix \ref{sA}. Assuming temporarily that $\supp \, (v) \subseteq [0,a]$, one derives with the help of \eqref{e50} and \eqref{w28},  
\begin{align}  
& \int_0^{\eT} e^{2i x z} \cF(x,\ell) \, dx \underset{\Im (z) \to \infty}{=} 
O\Big(z^2e^{2 i \eT z} \big/(\Im(z))^{1/2}\Big),    \lb{wf6}  \\   
& \cF(x, \ell):=\frac{\p}{\p \ell}\cla(x, \ell) - \frac{\p}{\p x}\cla(x,\ell)
- \int_0^x \cla(x-t,\ell)\cla(0,\ell)^*\cla(t,\ell) \, dt.    \lb{wf7} 
\end{align}
It is immediate from \eqref{wf6} that 
\begin{equation}  \lb{wf8}
\Up(z,\ell):= \int_0^{\eT} e^{2i (x-T) z} \cF(x,\ell) \, dx  \underset{|z| \to \infty}{=} O\big(z^2\big).
\end{equation}
According to \eqref{wf8} and to \cite[Vol.~II, Ch.~9, Section~42, Lemma~2]{Ma85}, the entire function 
$\Up(\dott,\ell)$ is, in fact, a polynomial, 
\begin{align} & \lb{wf9}
\Up(z,\ell)= a_2(\ell)z^2+a_1(\ell)z+a_0(\ell).
\end{align}
A comparison of the definition of $\Up(z,\ell)$ in \eqref{wf8} with expression \eqref{wf9} reveals that
\begin{align} & \lb{wf10}
\Up(z,\ell)\equiv 0_{m_2 \times m_1}.
\end{align}
Finally, relations \eqref{wf7}, \eqref{wf8}, and \eqref{wf10} imply  \eqref{e44}.

We recall that the right-hand side of \eqref{e44} is continuous with respect to $x$, and so $\frac{\p}{\p \ell}\cla(x, \ell)$
is continuous with respect to $x$ as well. 
Hence, since the matrix functions $\cla_k(x,\ell)$ are continuous in $(x,\ell)$, the functions $f_k(\a)$ given by \eqref{w4} are continuous, the equality
\begin{equation} 
\frac{\p}{\p \ell}\cla(x,\ell)=\sum_{k=0}^{\infty}\frac{\p}{\p \ell}\cla_k(x,\ell) 
\end{equation} 
is valid, 
equalities \eqref{w21} and \eqref{wf12} hold, and $\wh \om(\a,\eta)$ in \eqref{wf12} is continuous with respect to $\a$ (see Remark \ref{LR}),
one concludes that the matrix function  $\frac{\p}{\p \ell}\cla(x, \ell)$ is continuous with respect
to the pair  $x$ and $\ell$. Now, \eqref{e44} yields that $\frac{\p}{\p x}\cla(x,\ell)$ is also continuous with respect
to $(x,\ell)$.

It remains to remove the additional condition $\supp \, (v) \subseteq [0,a]$. Indeed, according to \eqref{e16}, \eqref{e42}, and Remark \ref{Rktilde2}, the function $\cla(x,\ell)$, where $x\leq b_1$ and $\ell \leq b_2$, is uniquely determined by $v(x)$ for $x \in [0,b_1+b_2]$.
Thus, one can indeed abandon the requirement $\supp \, (v) \subseteq [0,a]$ for \eqref{e44} to hold.
\end{proof}
%%%%%%

%%%%%%
%%%%%%
\section{The Inverse Approach} \lb{s6}
%%%%%%
%%%%%%

Similarly to the procedure for scalar  Schr\"odinger operators in \cite{GS00, Re03, Si99}, one can solve the  
inverse problem for the Dirac-type system by solving the $\cla$-equation \eqref{e44} with the boundary condition $\cla(x, 0)=\cla(x)$
(which  is easily recovered from the Weyl--Titchmarsh function by taking the inverse Fourier transform in \eqref{Repr})
and by using the equality $v(\ell)=-i \cla(0,\ell)^*$ in \eqref{e43}.
In order to demonstrate that this approach works, we next prove the existence and uniqueness results for \eqref{e44}. 

First, we consider the necessary conditions on $\cla(x, 0)$, required in our inverse problem. It is immediate from \eqref{e15} that for $x \in (0,\eT)$
one has
\begin{equation}
(K f)(x) - (K^* f)(x) = i \g(x) J \int_0^{\eT}\g(t)^* f(t) \, dt, \quad f \in \big[L^2((0,\eT))\big]^{m_2}. 
\end{equation}
Multiplying this equality by $E^{-1}$ from the left and by $(E^{-1})^*$ from the right and taking into account \eqref{i1}, \eqref{i3}
and \eqref{e16}, one obtains the operator identity,
\begin{align} 
& \cI S-S \cI^*=i \Pi J \Pi^*; \quad S:=E^{-1}(E^{-1})^*,    \lb{e51} \\
\begin{split} 
& \Pi\in \cB\Big(\BC^m, \, \big[L^2((0, \eT))\big]^{m_2}\Big),    \\
& (\Pi g)(x)=\big(E^{-1}\g\big)(x)g=\begin{bmatrix}\Phi(x) & I_{m_2}\end{bmatrix}g, 
\quad g \in \bbC^m.   \lb{e52} 
\end{split} 
\end{align}
According to \cite[Proposition~2.41]{SSR13}, there is a unique bounded operator $S$,
which satisfies the identity $AS-SA^*=i \Pi J \Pi^*$ in \eqref{e51}. This $S$ is strictly positive definite, $S > 0$  
(i.e., there exists $\varepsilon > 0$ such that $S \geq \varepsilon I_{[L^2((0, \eT))]^{m_2}}$)
and its integral kernel is expressed via $\cla(\dott)$. More precisely, one has 
\begin{align}  \lb{e53}
\begin{split} 
& (S f)(x) = (S_{\eT} f)(x) = f(x) - \int_0^{\eT}s(x,t) f(t) \, dt > 0, \quad f \in \big[L^2((0,\eT))\big]^{m_2},  \\ 
& s(x,t):=\int_0^{\min(x,t)}\cla(x-\xi)\cla(t-\xi)^* \, d\xi \quad ((x,t) \in (0,\eT) \times (0,\eT)), 
\end{split} 
\end{align}
with
\begin{equation}
S := S_{\eT} > 0 \quad (\text{for all } \eT > 0).    \lb{e53S}
\end{equation}
Hence, the following corollary of Proposition \ref{PnDif} holds:

%%%%%%
\begin{corollary} \lb{CyNesCond} 
Suppose that $v \in \big[C^1([0, \infty))\big]^{m_1 \times m_2}$ in the Dirac-type system \eqref{A2}. Then, 
$\cla = \Phi^{\prime} \in \big[C^1([0,\infty))\big]^{m_2 \times m_1}$ and the positive definiteness condition 
\eqref{e53S} holds for all $\eT \in [0,\infty)$, that is, $S = S_{\eT} > 0$ for all $\eT > 0$.
\end{corollary}
%%%%%%

A fundamental fact is that the converse statement is also valid. To prove it we need to introduce the operators $S_x$ in $\big[L^2((0,x))\big]^{m_2}$ 
$(x \in (0,T))$ in the same way as we introduce  $S_{\eT}$:
\begin{equation}
(S_x f)(y):= f(y) - \int_0^x s(y,t) f(t) \, dt, \quad S_x \in \cB\Big(\big[L^2((0,x))\big]^{m_2}\Big), 
\end{equation}
where $s(x,t)$ is given by the last equality in \eqref{e53}.

%%%%%%
\begin{proposition}\lb{PnSufCond} 
Let the $m_2 \times m_1$ matrix function $\cla \in \big[C^1([0, \infty))\big]^{m_2 \times m_1}$ satisfy the 
positive definiteness condition \eqref{e53S} for all $\eT \in (0,\infty)$. Then there exists 
$v \in \big[C^1([0, \infty))\big]^{m_1 \times m_2}$ in  
the Dirac-type system \eqref{A2} such that $\cla = \Phi^{\prime}$. In particular,
\begin{equation}
v(\ell) = - i \cla(0,\ell)^* \quad (\ell \geq 0).    \lb{6.v}
\end{equation}
\end{proposition}
%%%%%%
\begin{proof} It follows from the proofs of Theorem 2.54 and Lemma 2.55 in \cite{SSR13}  that the 
potential $v$ in the Dirac-type system \eqref{A2}, such that one has $\Phi^{\prime} = \cla$ (for the 
$m_2 \times m_1$ matrix-valued function $\Phi$ corresponding to this Dirac-type system), is given by the 
formula $v=i \b_{\Phi}^{\prime} J\g_{\Phi}$, where 
\begin{align} 
& \Phi(x) = \int_0^{x}\cla(t) \, dt,  \\
& \b_{\Phi}(x):=\begin{bmatrix}I_{m_1} &0_{m_1 \times m_2} \end{bmatrix}
+\int_0^x \big(S_x^{-1}\cla\big)(t)^*\begin{bmatrix}\Phi(t) & I_{m_2} \end{bmatrix} \, dt,    \lb{e54} \\
& \g_{\Phi}(x):=E_{\Phi}\begin{bmatrix}\Phi(x) & I_{m_2} \end{bmatrix}.      \lb{e55} 
\end{align}  
Here, the operators $E_{\Phi} \in \cB\Big(\big[L^2((0, \eT))\big]^{m_2}\Big)$ are of the form
\begin{equation}
(E_{\Phi} f)(x) = f(x) + \int_0^{x} E_{\Phi}(x,t) f(t) \, dt, \quad f \in \big[L^2((0,\eT))\big]^{m_2},
\end{equation}
with  continuous integral kernels $E_{\Phi}(\dott,\dott)$, and they are uniquely determined by the factorizations $S_{\eT}^{-1}=E_{\Phi}^*E_{\Phi}$,
where $S_{\eT}$ are given by \eqref{e53}. The operators $E_{\Phi}$  are applied (in \eqref{e55}) to the matrix function $\begin{bmatrix}\Phi(x) & I_{m_2} \end{bmatrix}$
columnwise.
We note that $E_{\Phi}(x,t)$ does not depend on the choice
of $\eT$ as long as $\eT \geq x$.

It remains to show that $v$ is continuously differentiable (since \cite[Theorem 2.54]{SSR13} deals with the case of locally bounded $\cla$ and $v$). For that purpose, it suffices to show that $\g$ is continuously differentiable, that $\b$
is two times continuously differentiable, and use $v=i \b_{\Phi}^{\prime} J \g_{\Phi}$.
Since $E_{\Phi}$ does not depend on $\eT$, one can factorize $S_x^{-1}$ and rewrite
\eqref{e54} in the form
\begin{align} &      \lb{e56}
\b_{\Phi}(x):=\begin{bmatrix}I_{m_1} &0_{m_1 \times m_2} \end{bmatrix}
+\int_0^x (E_{\Phi}\cla)(t)^* 
\big(E_{\Phi}\begin{bmatrix}\Phi(t) & I_{m_2} \end{bmatrix}\big) \, dt.
\end{align}  
In view of \eqref{e55} and \eqref{e56}, it suffices to prove that $E_{\Phi}$ maps continuously differentiable functions into continuously
differentiable functions.  The last property of $E_{\Phi}$ follows from the fact that $E_{\Phi}(x,t)$ is differentiable with respect to $x$
and both functions $E_{\Phi}(x,t)$ and $\frac{\p}{\p x}E_{\Phi}(x,t)$ ($x \geq t$) are continuous in $(x,t)$.
This fact will be proven below using some results in \cite[pp.~185--186]{GK70}.

The unique factorization $S_{\eT}^{-1}=E_{\Phi}^*E_{\Phi}$ coincides with 
the unique factorization in \cite{GK70}, that is, $E_{\Phi}(x,t)$ coincides with 
$V_-(x,t)$ in the notation of \cite{GK70}. Hence, according to formula (7.9) in \cite[p.~186]{GK70} one has 
\begin{align} &      \lb{e57}
E_{\Phi}(\xi,t)=\G_{\xi}(\xi,t) \quad (\xi \geq t),
\end{align}
where $\G_{\xi}(x,t)$ is the integral kernel of the operator $S_{\xi}^{-1}$, 
\begin{align} &      \lb{e58}
(S_{\xi}^{-1} f)(x) = f(x) + \int_0^{\xi}\G_{\xi}(x,t) f(t) \, dt, \quad f \in \big[L^2((0,\xi))\big]^{m_2}
\end{align}
(for $\xi \in (0,\eT)$). According to \cite[pp.~185--186]{GK70}, $\G_{\xi}(x,t)$ is continuous with respect to $(\xi, x, t)$ and formula (7.10) in \cite{GK70} holds 
\begin{align} &      \lb{e59}
\frac{\p}{\p \xi}\G_{\xi}(x,t)=\G_{\xi}(x,\xi)\G_{\xi}(\xi,t).
\end{align}
The first equality in eq. (7.7) in \cite{GK70} can then be rewritten in the form
\begin{align} &      \lb{e60}
\G_{\xi}(x,t)=s(x,t)+\int_0^{\xi}s(x,r)\G_{\xi}(r,t) \, dr.
\end{align}
In particular, one obtains 
\begin{align} &      \lb{e61}
\G_{\xi}(\xi,t)=s(\xi,t)+\int_0^{\xi}s(\xi,r)\G_{\xi}(r,t) \, dr, 
\end{align}
where $s(\, \dott, \, \dott \,)$ for our case is given in \eqref{e53}. Clearly, $\frac{\p}{\p \xi}s(\xi,t)$ is  continuous  with respect to $(\xi,t)$, $\xi \geq t$.
Hence, in view of \eqref{e59} and \eqref{e61},  $\frac{\p}{\p \xi}\G_{\xi}(\xi,t)$ is  continuous  with respect to $(\xi,t)$, 
$\xi  \geq t$. In other words (since \eqref{e57} is valid), $E_{\Phi}(x,t)$ is differentiable with respect to $x$
and  $\frac{\p}{\p x}E_{\Phi}(x,t)$ ($x \geq t$) is continuous with respect to $(x,t)$.
Hence, the potential $v$ recovered above belongs to $\big[C^1([0, \infty))\big]^{m_1 \times m_2}$.
Finally, relation \eqref{6.v} for the corresponding $\cla$-function
 is clear from \eqref{e43}.
\end{proof}
%%%%%%

Now, we can prove the existence and uniqueness result for solutions of  the $\cla$-equation, that is, for the system \eqref{e44}. Recalling our notation 
\begin{equation} 
\bbR^2_{+,+} = \big\{(x,\ell) \in \bbR^2 \, \big| \, x\geq 0, \, \ell \geq 0\big\} 
\end{equation} 
for the first quadrant, we have the following result. 

%%%%%%
\begin{theorem} \label{Tm6.3} The system \eqref{e44} with initial condition
\begin{equation}    \lb{e62}
\cla(x,0)=\cla(x) \in \big[C^1([0,\infty))\big]^{m_2 \times m_1},
\end{equation}
such that the positive definiteness property $S = S_{\eT} > 0$ $($cf.\ \eqref{e53S}$)$ holds for all $\eT \in (0,\infty)$, has a unique solution 
\begin{equation} 
\cla(\dott,\dott)\in \big[C^1(\bbR^2_{+,+})\big]^{m_2 \times m_1}. 
\end{equation} 
\end{theorem} 
%%%%%%
\begin{proof} The existence of the solution $\cla$ is immediate after we consecutively use Proposition \ref{PnSufCond}
and Theorem \ref{AeqfD}. 

The uniqueness is proved somewhat similar to the proof of \cite[Theorem 7.1]{Si99}.
Indeed, set $f(r,\ell)=\cla(r-\ell, \ell)$, then,  
\begin{align} &      \lb{e63}
\left(\frac{\p}{\p \ell}f\right)(r,\ell)=\left(\frac{\p}{\p \ell}\cla\right)(r-\ell, \ell)-\left(\frac{\p}{\p x}\cla\right)(r-\ell,\ell).
\end{align}
In view of \eqref{e44} and \eqref{e63}, one concludes that 
\begin{align}       
f(r, \ell_2)- f(r, \ell_1)&=\int_{\ell_1}^{\ell_2}\int_0^{r-\ell}
\cla(r-\ell-t,\ell)\cla(0,\ell)^*\cla(t, \ell) \, dt \, d\ell
\no \\ \lb{e64} &
=\int_{\ell_1}^{\ell_2}\int_0^{r-\ell}f(r-t,\ell)f(\ell,\ell)^*f(t+\ell, \ell) \, dt \, d\ell, 
\end{align}
where $r\geq \ell_2>\ell_1 \geq 0$. 

Next, we prove the uniqueness of the solution of \eqref{e44}, \eqref{e62}
by contradiction. Assuming that there are two solutions $\cla_1$ and $\cla_2$, 
we set $f_k(r, \ell)=\cla_k(r-\ell, \ell)$, $k=1,2$, and use \eqref{e64} in order to obtain
\begin{align}       \nn
& f_1(r, \ell_2)- f_2(r, \ell_2)  =
 f_1(r, \ell_1)- f_2(r, \ell_1)
\\ \nn & \quad
+\int_{\ell_1}^{\ell_2}\int_0^{r-\ell}\Big(\big(f_1(r-t,\ell)-f_2(r-t,\ell)\big)
f_1(\ell,\ell)^* f_1(t+\ell, \ell)
\\  \nn &\quad  
+f_2(r-t,\ell)\big(f_1(\ell,\ell)^*- f_2(\ell,\ell)^*\big)f_1(t+\ell, \ell)
\\  \lb{e65} & \quad 
+f_2(r-t,\ell)f_2(\ell,\ell)^*\big(f_1(t+\ell, \ell)-f_2(t+\ell, \ell)\big)\Big) \, dt \, d\ell .
\end{align}
Putting
\begin{align} &      \lb{e66}
F(a,\ell):=\max_{\ell\leq r \leq a} (\|f_1(r,\ell) -f_2(r,\ell)\|)
\end{align}
(with $\|\dott \|$ a convenient matrix norm), and taking into account \eqref{e65}, one derives 
\begin{align} &      \lb{e67}
F(a,\ell_2) \leq F(a,\ell_1)+\wt C \int_{\ell_1}^{\ell_2}(r-\ell) F(a, \ell) \, d \ell 
\leq F(a,\ell_1)+C \int_{\ell_1}^{\ell_2} F(a, \ell) \, d \ell
\end{align}
for any fixed $a>0$ and some $\wt C, \, C >0$. Our concluding arguments now coincide with the end of the
proof of \cite[Theorem 7.1]{Si99}: formula \eqref{e67} yields
\begin{equation}       \lb{e68}
\clf(a,\ell_2)\leq \clf(a, \ell_1)+C(\ell_2-\ell_1)\clf(a,\ell_2),  
\end{equation} 
where 
\begin{equation}
\clf(a, \ell):=\max_{0\leq s \leq \ell} (F(a,s)).
\end{equation}
According to \eqref{e62}, one concludes that $f_k(r,0)=\cla_k(r,0)=\cla(r)$, and so $F(a, 0)= 0_{m_2 \times m_1}$. Hence, $\clf(a, 0) = 0_{m_2 \times m_1}$.
Therefore, formula \eqref{e68} with $\ell_1=0$ (and $\ell_2=\ell$)  implies 
that $\clf(a, \ell) = 0_{m_2 \times m_1}$ for $\ell <1/C$.
Repeating this argument a finite number of times then yields $\clf(a, \ell) = 0_{m_2 \times m_1}$ 
for all $\ell \leq a$, that is,
$\cla_1(x, \ell)=\cla_2(x, \ell)$ in the triangle $x+\ell \leq a$. Since the equality holds for any $a>0$,
the solution $\cla (x,\ell)$ is unique. 
\end{proof}
%%%%%%

We refer to \eqref{1.17a} and \eqref{1.17b} for a succinct visual summary of the direct and, especially, the inverse approach developed in this paper. 

%%%%%%
\begin{remark} \lb{r6.4} For convenience of exposition, we considered the Dirac-type system \eqref{A2} 
on the half-axis and the $\cla$-equation \eqref{e44} on the first quadrant $\bbR^2_{+,+}$ 
(i.e., we focused on global solutions). However, our considerations are applicable to the local case as well. Indeed, given $\cla(x,0) \in \big[C^1([0, \eT])\big]^{m_2 \times m_1}$, we use the relations $v=i \b_{\Phi}^{\prime} J\g_{\Phi}$ and \eqref{e54}, \eqref{e55} in the proof of Proposition \ref{PnSufCond} in order to recover the 
Dirac-type system on [0,\eT] with a continuously differentiable potential $v \in \big[C^1([0,\eT])\big]^{m_1 \times m_2}$. Clearly, $v(\dott)$ may be easily extended to $v \in \big[C^1([0, \infty))\big]^{m_1 \times m_2}$. 
For each such extension (and associated Weyl--Titchmarsh function) the global solution of \eqref{e44} 
exists, and so the local  solutions exist as well. 
Moreover, according to the proofs of  Theorem~2.54 and Lemma~2.55 in \cite{SSR13}, the corresponding $\cla$-function
coincides on $[0,\eT]$ with the initial function $\cla(x,0)$. The uniqueness of the solution of \eqref{e44} is also proved locally, namely, for $\cla(x,\ell)$
in the triangle  $x+\ell \leq \eT$ $($see the proof of Theorem \ref{Tm6.3}$)$. To illustrate this point we briefly derive a version of the local Borg--Marchenko uniqueness theorem for Dirac-type systems \eqref{A2} next. \hfill $\diamond$
\end{remark}
%%%%%%

%%%%%%
\begin{theorem}\lb{t6.5}
Let $\vp(z)$ and $\wt \vp(z)$ be the Weyl--Titchmarsh functions of two Dirac-type systems
on the half-axis with corresponding potentials $v \in \big[C^1([0, \infty))\big]^{m_1 \times m_2}$ and 
$\wt v \in \big[C^1([0, \infty))\big]^{m_1 \times m_2}$, respectively.
Suppose that on some ray $\Re(z)=c \, \Im(z)$ $(c\in \BR$, $\Im(z)>0)$ for each $\ve>0$ one has $($again for a convenient 
matrix norm $\|\dott \|$$)$ 
\begin{equation}       \lb{eBM}
\big\|\vp(z)-\wt \vp(z)\big\| \underset{(|z| \to \infty}{=}O\Big(e^{2i(\eT-\ve)z}\Big).
\end{equation}
Then, 
\begin{equation}       \lb{eBM'}
v(x) =\wt v(x) \, \text{ for all } \, x\in [0,\eT].
\end{equation}
\end{theorem}
\begin{proof}  Using \eqref{e42}, \eqref{e45}, and Proposition \ref{PnDif} one integrates in \eqref{i27}
by parts and derives
\begin{equation}       \lb{BM1}
\vp(z)=-\int_0^{\eT-\ve}e^{2ixz}\cla(x,0) \, dx + O\Big(z e^{2i (\eT-\ve) z} \big/ (\Im(z))^{1/2}\Big) \quad (\eT>2\ve >0).
\end{equation}
Clearly, one has a similar formula for $\wt \vp(z)$, implying the relation
\begin{align}       \lb{BM2}
\vp(z)-\wt \vp(z)=&\int_0^{\eT-2\ve}e^{2ixz}\big[\wt \cla(x,0)- \cla(x,0)\big] \, dx + 
\\ & \nn +\int_{\eT-2\ve}^{\eT-\ve}e^{2ixz}\big[\wt \cla(x,0)- \cla(x,0)\big] \, dx + O\Big(z e^{2i (\eT-\ve) z} \big/ (\Im(z))^{1/2}\Big). 
\end{align}
Here, $\wt \cla$ is the $\cla$-function for Dirac-type system \eqref{A2} corresponding to the potential $\wt v$. Taking into account \eqref{eBM} and \eqref{BM2}, 
one infers that the entries of the $m_2 \times m_1$ matrix-valued function 
\begin{align}       \lb{BM2'}
W(z)=e^{-2i(\eT-2\ve)z}\int_0^{\eT-2\ve}e^{2ixz}\big[\wt \cla(x,0)- \cla(x,0)\big] \, dx
\end{align}
tend to zero on the ray $\Re(z)=c \, \Im(z)$. According to the definition \eqref{BM2'}, these entries  are also bounded on $\BC_-\cup \BR$
(in fact, they tend to zero).
Thus, applying the Phragmen--Lindel\"of theorem to the entries of $W(z)$
for the two sectors between our ray and the real line $\BR$, one derives that $W(z)$ is constant.
Moreover, since $W(z)$ tends to zero on some rays, one has
$W(z)\equiv 0_{m_2 \times m_1}$. Hence, $\cla(x,0)=\wt \cla(x,0)$
for all $x\in [0,\eT-2\ve]$ and all $\ve >0$, that is, for all
$x\in [0,\eT]$. Now, \eqref{eBM'} follows from the proof of the uniqueness in Theorem \ref{Tm6.3} (see also Remark \ref{r6.4}) and from  formula
\eqref{6.v}.
\end{proof}
%%%%%%

Several versions of the local Borg--Marchenko uniqueness result, Theorem \ref{t6.5}, exist in the literature under varying hypotheses on $v$. For the case of locally bounded rectangular matrices $v$ we refer to \cite{FKRS12}, \cite[Section~2.3.3]{SSR13}, for the case of locally square integrable rectangular matrics $v$, see \cite{Sa15}. A particular normal form of self-adjoint Dirac-type operators involving  locally integrable square matrices $v$ was considered in \cite{CG02}. The current proof of Theorem \ref{t6.5} based on the $\cla$-function concept distinguishes itself due to its particular simplicity. 

We emphasize once more that the results in this paper apply to matrix-valued Schr\"odinger operators as indicated in 
Remark \ref{r1.7}\,$(i)$.

%%%%%%%%%%%%%% appendices %%%%%%%%%%%%%%%%%
\appendix
%%%%%%%%%%%%%%% Appendix A %%%%%%%%%%%%%%%%
\section{Various Results in Support of Step 2 \\ in the Proof of Theorem \ref{AeqfD}} \lb{sA}
\renewcommand{\theequation}{A.\arabic{equation}}
\renewcommand{\thetheorem}{A.\arabic{theorem}}
\setcounter{theorem}{0} \setcounter{equation}{0}
%%%%%%%%%%%%%%%%%%%%%%%%%%%%%%%%%%%%%

We divide this appendix into three parts. \\

\paragraph{{\bf Part 1.}} We consider the Dirac-type  system  \eqref{A2} on $[0,\infty)$, setting
$u(x,z):=u_0(x,z)$ (implying $u_0(0,z) = I_m$) and $\vp(x,z):=\vp_0(x,z)$ (cf.\ \eqref{e4}), and introduce the differential
expression $\cL=-i J \frac{d}{dx}- V(x)$. For the fundamental solution 
$u(x,z)$ of the system \eqref{A2}, one has
\begin{align} &      \lb{A.1}
\cL u=z u.
\end{align}

Recalling that $\vp$ is the Weyl--Titchmarsh function of system \eqref{A2}, it is convenient to introduce also the fundamental solutions $w(x,z) $ with a normalization at $x=0$ different from $u(x,z)$:
\begin{align} &      \lb{A.2}
w(x,z)=u(x,z)Q(z), \quad Q(z):=\begin{bmatrix} 
I_{m_1} & 0_{m_1 \times m_2} \\ \vp(z) & I_{m_2}
\end{bmatrix}.
\end{align}
By $L$ we denote the operator acting in $\big[L^2((0,\infty))\big]^m$
via the differential expression $\cL$, such that functions $Y$ in the domain of $L$  satisfy the boundary condition
\begin{align} &      \lb{A.3}
\begin{bmatrix}
I_{m_1} & 0_{m_1 \times m_2} \end{bmatrix}Y(0)=0_{m_1 \times 1}.
\end{align}

A discussion of the resolvent of Dirac-type operators on an interval can be found, for instance, in \cite[Theorem 9.4.1]{At64}. 
In our case, the situation is similar and simple calculations
show that
\begin{align} &      \lb{A.4}
\cL Y_p=zY_p+f \quad {\mathrm{for}} \quad Y_p(x,z):=i w(x,z)\int_0^x w(t,z)^{-1} J f(t) \, dt,
\end{align}
where the assumption $f\in \big[L^2((0,\infty))\big]^m$ suffices for our purposes.
Formulas \eqref{A.1} and \eqref{A.4} imply that (for $g$ independent of $x \geq 0$) 
\begin{align} &      \lb{A.5}
\big(\cL -z I_{[L^2((0,\infty))]^{m}}\big)Y(x,z)=f(x) \, \text{ for } \, Y(x,z)=Y_p(x,z)+w(x,z)g(z).
\end{align}

Next, we assume that the supports of $V$ and $f$ belong to some finite interval $[0,a]$ for 
some $a > 0$, 
\begin{align} &      \lb{A.5'}
\supp \, (V) \subseteq [0,a], \quad \supp \,  (f) \subseteq [0,a].
\end{align}

We note that when $f$, $g$ and $Y$ are matrix-valued functions (instead of being vector-valued),
relations like $LY$ and $Y(\dott,z)\in \big[L^2((0,\infty))\big]^{m \times r}$ ($1 \leq r \leq m$)
are considered columnwise and the formulas above remain valid. The scalar product
$(Y_1, Y_2)$ in $\big[L^2((0,\infty))\big]^m$ is understood for matrix-valued as well as vector-valued functions as
\begin{equation} 
(Y_1, Y_2)_{[L^2((0,\infty))]^m} = \int_0^\infty Y_2(x)^*Y_1(x) \, dx.
\end{equation}

Next, we choose $g(z)$ in \eqref{A.5} so that the following conditions hold:
\begin{align} &      \lb{A.6}
Y(x,z)\in \big[L^2((0,\infty))\big]^m, \quad \begin{bmatrix}
I_{m_1} & 0_{m_1 \times m_2} \end{bmatrix}Y(0)=0_{m_1 \times 1}. 
\end{align}
Taking into account \eqref{A.4} and \eqref{A.5}, one rewrites $Y$
as follows:
\begin{align} &      \lb{A.7}
Y(x,z)=w(x,z)\left(g(z)+i \int_0^x w(t,z)^{-1} J f(t) \, dt\right).
\end{align}
Employing equations \eqref{A.5'} and \eqref{A.6}, one simplifies \eqref{A.7} for $x\geq a$
and fixes the required function $g(z)$ as follows 
\begin{align} &      \lb{A.8}
Y(x,z)=w(x,z)\left(g(z)+i \int_0^a w(t,z)^{-1} J f(t) \, dt\right)
\quad {\mathrm{for}} \quad x \geq a, 
\\ &      \lb{A.9}
g(z)=
\begin{bmatrix}
g_1(z) \\ g_2(z)
\end{bmatrix},
\quad
g_1(z)=0_{m_1}, \quad g_2(z)=-i \begin{bmatrix} 0_{m_2 \times m_1} &
I_{m_2} 
\end{bmatrix}\int_0^a w(t,z)^{-1} J f(t) \, dt.
\end{align}
Indeed, formulas \eqref{A.2}, \eqref{A.8}, and \eqref{A.9} show that $Y(x,z)$ for $x\geq a$ takes on the form
\begin{equation} 
Y(x,z)=u(x,z)\begin{bmatrix} 
I_{m_1}  \\ \vp(z) 
\end{bmatrix}c(z). 
\end{equation} 
Hence, in view of \eqref{e2}, the first condition in \eqref{A.6} is fulfilled. The second relation in \eqref{A.6}
follows from \eqref{A.7}, from the normalization $w(0,z)=Q(z)$ and from the equality $g_1(z)=0$.
It is easy to see that the requirements \eqref{A.9} are not only sufficient but also necessary for \eqref{A.6} to be valid.

Formulas \eqref{A.5} and \eqref{A.6} prove the following result.

%%%%%%
\begin{lemma} \lb{lA.1} 
Suppose that $f\in \big[L^2((0,\infty))\big]^m$ and $\supp \,  (f) \subseteq [0,a]$. Then
\begin{align} &      \lb{A.10}
(L-z I_{[L^2((0,\infty))]^{m}})^{-1}f=Y,
\end{align}
where $Y(x,z)$ is given by \eqref{A.7} and \eqref{A.9}.
\end{lemma}
%%%%%%

Taking into account \eqref{A2}, one obtains the equalities
\begin{align} &      \lb{A.11}
\frac{d}{dt}\big(w(t,z)^{-1}e^{i tz J}\big)=w(t,z)^{-1}(-i) J V(t) e^{i tz J},
\\[1mm]  &      \lb{A.12}
\frac{d}{dx}\big(e^{-i xz J}u(x,z)\big)=e^{-i xz J} i J V(x) u(x,z).
\end{align}
Equalities \eqref{A.10}--\eqref{A.12} are essential in the proof of the next proposition.

%%%%%%
\begin{proposition} \lb{pA.2} 
Suppose that $V$ satisfies the conditions $V\in \big[L^2((0,\infty))\big]^{m \times m}$ and 
$\supp \,  (V) \subseteq [0,a]$.
Then
\begin{align} &      \lb{A.13}
\vp(z)= i \int_0^ae^{2i x z}v(x)^*dx+ i \big((L-z I_{[L^2((0,\infty))]^{m}})^{-1}f_1, f_2\big), \\[1mm] 
&      \lb{A.14}
 f_1(x,z):=V(x)e^{i x z J}\begin{bmatrix}
I_{m_1} \\ 0
\end{bmatrix}, \quad f_2:=V(x)e^{i x \ov{z} J}\begin{bmatrix}
0 \\I_{m_2} 
\end{bmatrix}.
\end{align}
\end{proposition} 
%%%%%%
\begin{proof} 
Taking into account \eqref{A.7}, \eqref{A.10} and \eqref{A.11}, one rewrites $(L-zI)^{-1}f_1$, where
$f_1$ is given in \eqref{A.14}, in the form
\begin{align} &      \lb{A.15}
(L - z I_{[L^2((0,\infty))]^{m}})^{-1}f_1 = \big(u(x,z)-e^{i x z J}\big)\begin{bmatrix}
I_{m_1} \\ 0_{m_2 \times m_1}
\end{bmatrix}+ w(x,z)g(z).
\end{align}
Next, in view of \eqref{A.9}, \eqref{A.11}, and \eqref{A.14}, one derives
\begin{align} &      \lb{A.16}
g(z)=\begin{bmatrix}
0_{m_1 \times m_2} \\ I_{m_2} 
\end{bmatrix}\begin{bmatrix}
0_{m_2 \times m_1}  & I_{m_2} 
\end{bmatrix}\big(w(a,z)^{-1}e^{i a z J}- Q(z)^{-1}\big)
\begin{bmatrix}
I_{m_1} \\ 0_{m_2 \times m_1}
\end{bmatrix}.
\end{align}
Using \eqref{A.2}, one rewrites \eqref{A.16} as
\begin{align}  \nonumber    
g(z)&=\begin{bmatrix}
0_{m_1 \times m_2} \\ I_{m_2} 
\end{bmatrix}\begin{bmatrix}
-\vp(z)  & I_{m_2} 
\end{bmatrix}\big(u(a,z)^{-1}e^{i a z J}- I_m\big)
\begin{bmatrix}
I_{m_1} \\ 0_{m_2 \times m_1}
\end{bmatrix}
\\[1mm]   & \lb{A.17}
=\begin{bmatrix}
0_{m_1 \times m_2} \\ I_{m_2} 
\end{bmatrix}\left(\vp(z)+\begin{bmatrix}
-\vp(z)  & I_{m_2} 
\end{bmatrix}u(a,z)^{-1}e^{i a z J}\begin{bmatrix}
I_{m_1} \\ 0_{m_2 \times m_1}
\end{bmatrix}\right).
\end{align}
Since $\supp \, (V) \subseteq [0,a]$, one has $u(x,z)=e^{i (x-a)z J}u(a,z)$ for $x \geq a$. 
Therefore, one  obtains the relation
\begin{equation} 
u(x,z)\begin{bmatrix}
I_{m_1} \\ \vp(z)
\end{bmatrix}=e^{i (x-a)z J}u(a,z)\begin{bmatrix}
I_{m_1} \\ \vp(z)
\end{bmatrix} \in \big[L^2((0,\infty))\big]^m, 
\end{equation} 
implying that $\begin{bmatrix}
0_{m_2 \times m_1} & I_{m_2} 
\end{bmatrix}u(a,z)\begin{bmatrix}
I_{m_1} \\ \vp(z)
\end{bmatrix}=0_{m_2 \times m_1}$ (in view of $\Im(z)>0$). That is,
\begin{align} &      \lb{A.18}
\vp(z)=-u_{22}(a,z)^{-1}u_{21}(a,z), \quad \begin{bmatrix}
-\vp(z)  & I_{m_2} 
\end{bmatrix}=u_{22}(a,z)^{-1}\begin{bmatrix}
u_{21}(a,z)  & u_{22}(a,z) 
\end{bmatrix}.
\end{align}
Formula \eqref{A.17} and the second equality in \eqref{A.18} yield
\begin{align} &      \lb{A.19}
g(z)=\begin{bmatrix}
0_{m_1} \\ \vp(z)
\end{bmatrix}.
\end{align}
Finally, from \eqref{A.15} and \eqref{A.19} one concludes that
\begin{align} &      \lb{A.20}
(L-z I_{[L^2((0,\infty))]^{m}})^{-1}f_1 = u(x,z)\begin{bmatrix}
I_{m_1} \\ \vp(z)
\end{bmatrix}-e^{i x z J}\begin{bmatrix}
I_{m_1} \\ 0_{m_2 \times m_1}
\end{bmatrix}.
\end{align}
By virtue of \eqref{A.14} and \eqref{A.20} one has
\begin{align} 
\big((L-z I_{[L^2((0,\infty))]^{m}})^{-1}f_1, f_2\big) &= \begin{bmatrix}
0_{m_2 \times m_1} & I_{m_2}
\end{bmatrix}\int_0^ae^{-i x z J}V(x)u(x,z) \, dx\begin{bmatrix}
I_{m_1} \\ \vp(z)
\end{bmatrix}   \no \\
& \quad -\int_0^ae^{2i x z}v(x)^* dx.    \lb{A.23}
\end{align}
Employing \eqref{A.12}, one simplifies \eqref{A.23}  as follows, 
\begin{align}      \nonumber
& \big((L-z I_{[L^2((0,\infty))]^{m}})^{-1}f_1, f_2\big)= i \begin{bmatrix}
0_{m_2 \times m_1} & I_{m_2}
\end{bmatrix}\big(e^{-i  a z J}u(a,z)-I_m\big)\begin{bmatrix}
I_{m_1} \\ \vp(z)
\end{bmatrix}      \no \\ 
& \qquad -\int_0^ae^{2i x z}v(x)^*dx     \no 
\\[1mm]   
& \quad 
=-i \vp(z)+i e^{i az}\begin{bmatrix}
u_{21}(a,z) & u_{22}(a,z)\end{bmatrix}\begin{bmatrix}
I_{m_1} \\ \vp(z)
\end{bmatrix}-\int_0^ae^{2i x z}v(x)^*dx.
\end{align}
Finally, applying the first equality in \eqref{A.18} one derives
\begin{align}     &\lb{A.24}
\big((L-z I_{[L^2((0,\infty))]^{m}})^{-1}f_1, f_2\big)=-i \vp(z)-\int_0^ae^{2i x z}v(x)^*dx,
\end{align}
which is \eqref{A.13}.
\end{proof}
%%%%%%

Since 
\begin{equation}
e^{i x z J } 
= \begin{bmatrix} e^{i x z } I_{m_1} & 0_{m_1 \times m_2} \\ 0_{m_2 \times m_1} & e^{- i x z } I_{m_2} \end{bmatrix},    \lb{A.25} 
\end{equation}
\eqref{A.14} is equivalent to
\begin{equation}
 f_1(x,z) = e^{i x z } \begin{bmatrix} 0_{m_1} \\ v(x)^*
\end{bmatrix}, \quad f_2 = e^{- i x \ol{z} } \begin{bmatrix} v(x) \\ 0_{m_2} \end{bmatrix}.    \lb{A.26} 
\end{equation}

In the special case $V = 0_{m \times m}$ a.e., we write $\cL_0$ and $L_0$ instead of $\cL$ and $L$ and note 
that formulas  \eqref{A.10} and \eqref{A.7} yield
\begin{align} 
& \big((L_0 - z I_{[L^2((0,\infty))]^{m}})^{-1} f\big)(x)=\int_0^{\infty} G_0(z,x,x') f(x') \, dx',     \\
& G_0(z,x,x') = i e^{iz |x - x'|} \begin{cases}
\begin{bmatrix} I_{m_1} & 0_{m_1 \times m_2} \\ 0_{m_2 \times m_1} & 0_{m_2} \end{bmatrix}, & x > x',  
\\[5mm] 
\begin{bmatrix} 0_{m_1} & 0_{m_1 \times m_2} \\ 0_{m_2 \times m_1} & I_{m_2} \end{bmatrix}, & x < x', 
\end{cases}   \quad  \Im(z) > 0.   \lb{A.27} 
\end{align}

Moreover, according to \eqref{A.27}, the Neumann series
\begin{equation}
(L - z I_{[L^2((0,\infty))]^{m}})^{-1} = \sum_{n=0}^{\infty} (L_0 - z I_{[L^2((0,\infty))]^{m}})^{-1} 
\big[V (L_0 - z I_{[L^2((0,\infty))]^{m}})^{-1}\big]^n   \lb{A.28} 
\end{equation}
is norm convergent on the subspace $\big[L^2((0,a))\big]^m$ for $0 < \Im(z)$ sufficiently large, and we now employ it to shed additional light on the Weyl--Titchmarsh function 
$\varphi$ with the help of \eqref{A.13}. 

%%%%%%
\begin{lemma} \lb{lA.3}
Suppose that $V$ satisfies $V\in \big[L^2((0,\infty))\big]^{m \times m}$ and 
$\supp \, (V) \subseteq [0,a]$. Then for $0 < \Im(z)$ sufficiently large,  
\begin{align}    
\begin{split} 
& \vp(z)= i \int_0^a  e^{2 i  x z}v(x)^*dx+ i \big((L-z I_{[L^2((0,\infty))]^{m}})^{-1}f_1, f_2\big)    \\
& \hspace*{7mm} = - \sum_{k=0}^{\infty} M_{2k+1}(z; V),    \lb{A.29} 
\end{split}
\end{align}
where
\begin{align}
 M_1(z; V) &= - i \int_0^a e^{2 i z x_0} v(x_0)^* \, dx_0 := \int_0^a 
e^{2 i z \alpha}  \cA_1(\alpha) \, d\alpha,   \no \\     
 M_{2k+1} (z; V)  &= (-1)^{k+1} i \int_0^a dx_{2k} \int_0^{x_{2k}} dx_{2k-1} \int_{x_{2k-1}}^a dx_{2k-2} \cdots  
\int_0^{x_2} dx_{1}\int_{x_{1}}^a dx_{0}     \no \\
& \quad \times e^{2 i z \a} 
v(x_{2k})^* v(x_{2k-1}) v(x_{2k-2})^* \cdots v(x_{1}) v(x_{0})^* ,   \lb{A.30} 
\\ \nn 
\a &= x_{2k}-x_{2k-1}+x_{2k-2}-\ldots -x_1+x_0.
\end{align}
\end{lemma}
%%%%%%
\begin{proof}
Insertion of the Neumann series \eqref{A.28} into formula \eqref{A.13} for $\varphi$, exploiting the explicit form of $G_0(z,\, \dott \,, \dott \,)$ in \eqref{A.27}, yields \eqref{A.29} and \eqref{A.30}. In particular, the special block matrix structure of $V$ and $(L_0 - z I)^{-1}$, and the form of $f_j$, $j = 1,2$, shows that all even terms vanish identically 
in the Neumann expansion inserted into the first line of \eqref{A.29}, 
\begin{equation}
\Big((L_0 - z I_{[L^2((0,\infty))]^{m}})^{-1} \big[V (L_0 - z I_{[L^2((0,\infty))]^{m}})^{-1}\big]^{2k} f_1, f_2\Big) = 0, \quad k \in \bbN_0.  
\end{equation}
\end{proof}
%%%%%%

In addition, we introduce the notation 
\begin{align}& \nn
 \ell_0 = \a=x_{2k}-x_{2k-1}+x_{2k-2} \ldots +x_2 -x_1+x_0, \\ &\nn
  \ell_1=x_{2k}-x_{2k-1}+x_{2k-2}\ldots +x_2 -x_1, 
\\ & \lb{ad1}   
   \ell_2=x_{2k}-x_{2k-1}+x_{2k-2}\ldots +x_2,
 \quad \ldots
\end{align}

\paragraph{{\bf Part 2.}}
Changing the order of integration in \eqref{A.30} one may derive the following lemma.

%%%%%%
\begin{lemma} \lb{lA.4} Suppose that $V$ satisfies $V\in \big[L^2((0,\infty))\big]^{m \times m}$ and 
$\supp \, (V) \subseteq [0,a]$. Then $M_{2k+1}$
admits representation
\begin{align}& \lb{w1}
M_{2k+1} (z; V) =  \int_0^{(k+1)a} e^{2 i z \alpha}  \cA_k(\alpha) \, d\a.
\end{align}
Moreover, for continuous $v$ such that
\begin{align}& \lb{w2}
\max\|v(x)\|\leq c \quad (c \in (0, \infty))
\end{align}
one obtains 
\begin{align}& \lb{w3}
\| \cA_k(\a)\|\leq c^{2k+1} \, f_k(\a),
\end{align}
where $f_k(\a)$ $($which does not depend on $c)$  is given by
\begin{align}& \lb{w4}
f_k(\a)={(-1)^{k+1}}\wh \cA_k(\a),
\end{align}
and $\wh  \cA_k(\a)$ is $ \cA_k(\a)$ for the special case
\begin{align} \lb{w5}&
v(x)=\wh v(x):= \begin{cases} i, & x \in [0, a], \\
0, & x \in(a, \infty).
\end{cases}
\end{align}
\end{lemma} 
%%%%%%
\begin{proof} 
Recalling \eqref{ad1} and putting $x_0=\a - \ell_1$, one gets 
\begin{align} 
\begin{split} 
\Lam(\a - \ell_1, x_1, \ldots, x_{2k})=&\Lam(x_0, x_1, \ldots, x_{2k})
\\   \lb{ww1}
= & i v(x_{2k})^* v(x_{2k-1}) v(x_{2k-2})^* \cdots v(x_{1}) v(x_{0})^*.
\end{split} 
\end{align}
Next, changing variables and the order of integration in \eqref{A.30} one obtains 
\begin{align} \nn &
\int_0^{x_2}dx_1\int_{x_1}^a dx_0 \, e^{2 i z \alpha}\Lam(x_0, x_1, \ldots, x_{2k})
\\ \nn &
=\int_{\ell_2}^{a+\ell_3}d\a \int_0^{x_2}dx_1\, e^{2 i z \alpha}\Lam(\a - \ell_1, x_1, \ldots, x_{2k})
\\   \lb{ww2} & \quad
+\int_{a+\ell_3}^{a+\ell_2}d\a \int_0^{a+\ell_2-\a}dx_1\, e^{2 i z \alpha}\Lam(\a - \ell_1, x_1, \ldots, x_{2k}).
\end{align}
One can see by induction that the change of order of integration with respect to $\a$ and with respect to 
$x_{2s}$, $x_{2s+1}$ generates integrals of the form $\int_{(r-1)a+\ell_{2s+2}}^{ra+\ell_{2s+3}}d\a$
or $\int_{ra+\ell_{2s+3}}^{ra+\ell_{2s+2}}d\a$, which follows from \eqref{ww2} and the next two
equalities:
\begin{align}\nn &
\int_0^{x_{2s+2}}dx_{2s+1}\int_{x_{2s+1}}^a dx_{2s}\int_{(r-1)a+\ell_{2s}}^{ra+\ell_{2s+1}}d\a
\\ \nn &
=\int_{(r-1)a+\ell_{2s+2}}^{ra+\ell_{2s+3}}d\a \int_0^{x_{2s+2}}dx_{2s+1}\int_{x_{2s+1}}^{\a-(r-1)a-\ell_{2s+1}} dx_{2s}
\\ \lb{ww3} & \quad 
+ \int_{ra+\ell_{2s+3}}^{ra+\ell_{2s+2}}d\a \int_0^{ra+\ell_{2s+2}-\a}dx_{2s+1}\int_{x_{2s+1}}^{\a-(r-1)a-\ell_{2s+1}} dx_{2s}
\end{align}
and 
\begin{align}\nn &
\int_0^{x_{2s+2}}dx_{2s+1}\int_{x_{2s+1}}^a dx_{2s}\int_{ra+\ell_{2s+1}}^{ra+\ell_{2s}}d\a
\\ \nn &
=\int_{ra+\ell_{2s+3}}^{ra+\ell_{2s+2}}d\a \int_{ra+\ell_{2s+2}-\a}^{x_{2s+2}}dx_{2s+1}\int_{x_{2s+1}}^{a} dx_{2s}
\\ \nn & \quad 
+ \int_{ra+\ell_{2s+2}}^{(r+1)a+\ell_{2s+3}}d\a \int_0^{x_{2s+2}}dx_{2s+1}\int^{a}_{\a-ra-\ell_{2s+1}} dx_{2s}
\\ \lb{ww4} & \quad 
+ \int_{(r+1)a+\ell_{2s+3}}^{(r+1)a+\ell_{2s+2}}d\a \int_0^{(r+1)a+\ell_{2s+2}-\a}dx_{2s+1}\int^{a}_{\a-ra-\ell_{2s+1}} dx_{2s}.
\end{align}
The last changes of order of integration in \eqref{A.30} are given by one of the equalities
\begin{align}& \lb{ww5}
\int_0^a dx_{2k}\int_{ra}^{ra+x_{2k}}d\a=\int_{ra}^{(r+1)a}d\a \int_{\a-ra}^a dx_{2k} \quad (r \leq k),
\\ & \lb{ww6}
\int_0^a dx_{2k}\int_{(r-1)a+x_{2k}}^{ra}d\a=\int_{(r-1)a}^{ra}d\a \int^{\a-(r-1)a}_0 dx_{2k} \quad (r \leq k).
\end{align}
The representation \eqref{w1} follows from \eqref{A.30} and from \eqref{ww2}--\eqref{ww6}.
Taking also equations \eqref{w2} and \eqref{ww1} into account, one derives \eqref{w3}.
\end{proof}
%%%%%%%%%%%%%%%%

%%%%%%%%%
\begin{remark} We note that since $\|v(x)\|$ is bounded (i.e., \eqref{w2}
holds), we use in our estimates for
$\cla_k$ and in the proof of the convergence
of the corresponding series for $\cla$ several results
on Dirac-type system with the special scalar $v = \widehat v$ given
by \eqref{w5}. These results provide the necessary estimates on the iterated
integrals, which appear in our calculations, and we resorted to this trick so the matrix structure in all 
quantities involved does not unnecessarily obscure the essential estimates. \hfill $\diamond$
\end{remark}
%%%%%%%%%

For our estimates we need the Weyl--Titchmarsh function $\wt \vp(z,c)$ of the system \eqref{A2}
with $v=c\wh v$, where $\wh v$ is  given by \eqref{w5}. In the following lemma we abandon the requirement $c>0$ and consider $c\in \BC$,  $\, c \neq 0$ because the proof in this more general case
remains the same and the example is of a certain independent interest.

%%%%%%
\begin{lemma} \lb{lA.5} Let $c\in \BC$,  $\, c \neq 0$, and consider the special case where 
\begin{align}& \lb{w5+}
v(x)=\wt v(x): =c\wh v(x) = \begin{cases} c \, i, & x \in [0, a], \\ 0, & x \in (a,\infty), \end{cases}
\end{align}
in the system \eqref{A2} on $[0, \infty)$. Then the associated Weyl--Titchmarsh function $\wt \vp(z,c)$ has the form
\begin{align}& \lb{w6}
\wt \vp(z,c)=\psi(z, \la)\, \Big|_{\la=|c|^2}, \quad \psi (z,\la)=i \ov{c}\frac{1-e^{2iq(z,\la)}}{z+q(z,\la)-(z-q(z,\la))e^{2iq(z,\la)}},
\\ & \lb{w7}
q(z,\la): = (z^2-\la)^{1/2} \quad \big(\Im( z)>0, \; \Im\big((z^2-\la)^{1/2}\big)>0\big).
\end{align}
\end{lemma}
%%%%%%
\begin{proof} Introducing the $2 \times 2$ matrix-valued functions
\begin{align}& \lb{w8}
D(z)=\diag\{q(z, |c|^2), \, -q(z, |c|^2)\}, \quad K(z):=\begin{bmatrix} -ic & -ic \\ z-q(z, |c|^2) & z+q(z, |c|^2) \end{bmatrix}, 
\end{align}
a direct calculation shows that
\begin{align}& \lb{w9}
 K(z)D(z)K(z)^{-1}=\begin{bmatrix} z & ic \\ i\ov{c} & -z \end{bmatrix}.
\end{align}
It follows that in the case of $v = \widetilde v$ given by \eqref{w5+} the fundamental solution $u(x,z)$ of \eqref{A2} 
is given by the formula
\begin{align}& \lb{w10}
 u(x,z)=\wt u(x,z):=\begin{cases} K(z)e^{ixD(z)}K(z)^{-1}, & x\leq a, \\[1mm] 
e^{i(x-a)zJ} K(z)e^{iaD(z)}K(z)^{-1}, & x> a.
\end{cases}
\end{align}
In view of \eqref{w10} and the inequality $\Im(z)>0$, one concludes 
that the relation 
\begin{align}& \lb{w11}
\big(K(z)e^{iaD(z)}K(z)^{-1}\big)_{21}+\big(K(z)e^{iaD(z)}K(z)^{-1}\big)_{22}\wt \vp(z,c)=0
\end{align}
(for the lower entries of $K(z)e^{iaD(z)}K(z)^{-1}$) yields the property
\begin{align}\lb{w12}&
\wt u(\dott ,z)\begin{bmatrix} 1 \\ \wt \vp(z,c)\end{bmatrix} \in \big[L^2([0, \infty))\big]^{2},
\end{align}
which is characteristic for the Weyl--Titchmarsh function (cf.\ \eqref{e2}). In other words,
one has 
\begin{align}& \lb{w13}
\wt \vp(z,c)=-\big(K(z)e^{iaD(z)}K(z)^{-1}\big)_{21}\Big/ \big(K(z)e^{iaD(z)}K(z)^{-1}\big)_{22}.
\end{align}
Explicit calculations then show that \eqref{w13} is equivalent to \eqref{w6}.
\end{proof}
%%%%%%
In the following we return to the special case $c>0$.
Assuming, for instance, that $\Im(z)\geq 2^{1/2} c$, one verifies that
$q(z,\la)$ is well-defined by \eqref{w7} and is analytic with respect to $\la$ in the disk
$|\la| \leq |c|^2+\ve$ for some $\ve>0$, which does not depend on $z$. Next, one 
rewrites the denominator in \eqref{w6} in the form
\begin{align}& \lb{w14}
{z+q(z,\la)-(z-q(z,\la))e^{2iq(z,\la)}}=(z+q(z,\la))\left(1-\frac{\la}{(z+q(z,\la))^2}e^{2iq(z,\la)}\right).
\end{align}
Since $\Im(z)\geq 2^{1/2} c$, this implies $|(z+q(z,\la))^2|>2|c|^2$. Hence, according to \eqref{w14}, 
the denominator of $\psi$ does not vanish in the disk $|\la| \leq |c|^2+\ve$, and
$\psi$ is analytic with respect to $\la$ in this disk. Therefore, $\psi(z,\la)$ admits a Taylor
expansion, which (taking into account  \eqref{w6}) we compare with the expansion \eqref{A.29}, \eqref{A.30}
for $\wt \vp$. It follows that
\begin{align}& \lb{w15}
\frac{1}{k!}\frac{\p^k \psi(z,\la)}{\p \la^k}\Big|_{\la=0}c^{2k}=-\wt M_{2k+1}(z,c),
\end{align}
where $\wt M_{2k+1}(z,c)$ is $M_{2k+1}(z,V)$ with $v = \widetilde v$ given by \eqref{w5+}. In view of \eqref{w1}, \eqref{w4}
and \eqref{w15} one derives 
\begin{align}& \lb{w16}
\frac{1}{c} \frac{(-1)^{k}}{k!}\frac{\p^k \psi(z,\la)}{\p \la^k}\Big|_{\la=0}=\int_0^{(k+1)a} e^{2iz\a}f_k(\a) \, d\a.
\end{align}
Since $\psi(z, \la)$ admits Taylor expansion in the disk $|\la| \leq |c|^2+\ve$, one concludes 
that $\sum_{k=0}^\infty \frac{1}{c}  \frac{(-1)^{k}}{k!}\frac{\p^k \psi(z,\la)}{\p \la^k}\Big|_{\la=0}c^{2k}$ converges absolutely, and, moreover, 
\begin{align}& \lb{w17}
\sum_{k=0}^\infty \frac{1}{c} \frac{(-1)^{k}}{k!}\frac{\p^k \psi(z,\la)}{\p \la^k}\Big|_{\la=0}c^{2k}=\psi(z, -c^2)\big/ c.
\end{align}
Relations \eqref{w6} and \eqref{w14} imply also that $\psi (\xi+i\eta, -c^2)\big/c\in L^2(\bbR; d\xi)$ (slightly abusing notation for the sake of simplicity) for all sufficiently large fixed values of $\eta$.

%%%%%
\begin{remark}\lb{RkL} 
We recall that in the case of bounded $v$ (more precisely in the case where \eqref{w2} is valid)
formula \eqref{A.29} holds for $\Im(z) \geq \frac{1}{2}ac^2$ and that in the considerations after the proof of
Lemma \ref{lA.5} we assumed that $\Im(z) \geq 2^{1/2} c$. Thus, \eqref{A.29}  and the statements and relations  after the proof of
Lemma \ref{lA.5}  hold if 
\begin{align}& \lb{w19-}
\eta = \Im(z) \geq C=\max \, \big(2^{1/2} c, \,\, ac^2/2\big).
\end{align}
When \eqref{w19-} is valid,
it follows from \eqref{w16} and \eqref{w17} that 
\begin{align}& \lb{wf1}
\psi (\xi+i\eta, -c^2)\big/c=\lim_{N\to \infty}\int_0^{(N+1)a} e^{2iz\a} \left(\sum_{k=0}^N f_k(\a)c^{2k}\right) d\alpha.
\end{align}
Formula \eqref{wf1} yields that the function $\psi (\xi+i\eta, -c^2)\big/c$ is positive definite (as a function of $\xi$)
for each fixed $\eta$ satisfying \eqref{w19-}. Furthermore, according to the properties of the positive definite functions (see, e.g., \cite{Lu72}
and more references in \cite{GP17})
the derivative $\om(\a, \eta)$ of the absolutely continuous part
of the distribution for the positive definite function $\psi (\xi+i\eta, -c^2)\big/c$  satisfies
the following relations
\begin{align}& \lb{wf2}
\om(\a, \eta)\in L^1(\bbR; d \alpha), \quad 2 \om(\a, \eta)\geq e^{-\eta \a}\sum_{k=0}^{\infty}c^{2k} f_k(\a/2) \quad (\a \geq 0).
\end{align}
Using \eqref{w3} and \eqref{wf2}, one concludes that $\sum_{k=0}^{\infty} \cla_k(\a)$ converges on each finite interval 
(in the $L^1$-norm) and that 
\begin{align}& \lb{wf3}
e^{- 2\eta \a}\sum_{k=0}^{\infty} \cla_k(\a)\in \big[L^1((0, \infty); d \alpha)\big]^{m_2 \times m_1}.
\end{align}
Since relations \eqref{A.29} and \eqref{w1} imply 
\begin{align}& \lb{w20}
\cA(\a)=\sum_{k=0}^{\infty}\cA_k(\a),
\end{align}
one finally concludes that
\begin{align}& \lb{wf4}
e^{- 2\eta \a} \cla(\a)\in \big[L^1((0, \infty); d \alpha)\big]^{m_2 \times m_1}.
\end{align}
${}$ \hfill $\diamond$
\end{remark}
%%%%%%%

\paragraph{{\bf Part 3.}} Next, we consider in greater detail the case 
$v\in \big[C^1([0,\infty))\big]^{m_1 \times m_2}$ (as in Theorem \ref{AeqfD}),
$\supp \, (v) \subseteq [0,a]$. Then, for each $v$ there are some values $c$, $\wh c$ such that
\begin{align}& \lb{w19}
\max(\|v(x)\|)<c, \quad \max(\|v^{\prime}(x)\|)<\wh c.
\end{align}
One recalls that we switch from $\vp(z)$ and $\cA(\a)$ to $\vp_{\ell}(z)$ and $\cA(x, \ell)$
by considering $v(x+\ell)$ instead of $v(x)$. Hence, \eqref{A.30}, \eqref{w1} and \eqref{w19} imply that 
\begin{align}& \lb{w21}
\bigg\| \frac{\p}{\p \ell}\cA_k(\a, \ell)\bigg\|\leq (2k+1)\wh c \, c^{2k} \, f_k(\a)
\end{align}
(cf.\ \eqref{w3}). On the other hand, similarly to \eqref{w17}, one derives
\begin{align}& \lb{w22}
\sum_{k=1}^\infty \frac{(-1)^{k}}{(k-1)!}\frac{\p^k \psi(z,\la)}{c\, \p \la^k}\Big|_{\la=0}c^{2k-2}=-\frac{1}{c}\frac{\p \psi(z, \la)}{\p \la}\Big|_{\la=-c^2},
\end{align}
and the series in \eqref{w22} converges absolutely. Taking into account \eqref{w16} and \eqref{w22} one can see
that the function $-\frac{1}{c}\frac{\p \psi(z, \la)}{\p \la}\Big|_{\la=-c^2}$ is positive definite. Moreover, according to the properties
of the positive definite functions \cite{Lu72} and because of relations \eqref{w16} and \eqref{w22} one 
has (similarly to \eqref{wf2})  
\begin{align}& \lb{wf5}
\wh \om(\a, \eta)\in L^1(\bbR; d \alpha), \quad 2 \wh \om(\a, \eta)\geq e^{-\eta \a}\sum_{k=1}^{\infty}k c^{2k} f_k(\a/2),
\end{align}
where $\wh \om(\a, \eta)$ is the derivative  of the absolutely continuous part
of the distribution for the  function $-\frac{1}{c}\frac{\p \psi(z, \la)}{\p \la}\Big|_{\la=-c^2}$. 
This implies 
that  the sum
\begin{equation} 
\sum_{k=0}^{\infty} (2k+1)\wh c \, c^{2k} f_k(\a) 
\end{equation} 
converges in $L^1((0,\infty); d \alpha)$ and that 
\begin{align}& \lb{w25}
e^{-\eta \a}\sum_{k=0}^{\infty}(2k+1)\wh c \, c^{2k} f_k(\a/2)\in L^1((0, \infty); d \alpha) \quad  (\eta \geq C).
\end{align}
Finally, by virtue of \eqref{w20}, \eqref{w21} and \eqref{w25}, one infers that $\cA(\a, \ell)$ is differentiable with respect to $\ell$
and
\begin{align}& \lb{w26}
e^{-\eta \a} \frac{\p}{\p \ell}\cA(\a, \ell)\in L^1((0, \infty); d \alpha)^{m_2 \times m_1} \quad  (\eta \geq 2C).
\end{align}
Hence, relations \eqref{Repr}, \eqref{e42} and \eqref{e45} yield the following lemma.

%%%%%%%
\begin{lemma}\lb{ML}
Suppose that $v$ satisfies $v\in \big[C^1([0,\infty))\big]^{m_1 \times m_2}$ and 
$\supp \, (v) \subseteq [0,a]$. Then, the equality
\begin{align} & \lb{w27}
\frac{d}{d \ell}\vp_{\ell}(z)=- \int_0^{\infty} e^{2i \a z} \frac{\p}{\p \ell}\cA(\a, \ell) \, d\a
\end{align}
is valid for $\Im(z) \geq C$ $($cf.\ \eqref{w19-}$)$.
\end{lemma}
%%%%%%%

According to \eqref{w26} and \eqref{w27} one has the following result.

%%%%%%%
\begin{corollary}\lb{MC} Suppose that $v$ satisfies $v\in \big[C^1([0,\infty))\big]^{m_1 \times m_2}$ and 
$\supp \, (v) \subseteq [0,a]$. Then
\begin{align} & \lb{w28}
\frac{d}{d \ell}\vp_{\ell}(z) \underset{\Im(z)\to \infty}{=} 
- \int_0^{\eT} e^{2i x z} \frac{\p}{\p \ell}\cA(x, \ell) \, dx+ O\Big(e^{2 i \eT z}\Big).
\end{align}
\end{corollary}
%%%%%%%

Finally, we study $-\frac{1}{c}\frac{\p \psi(z, \la)}{\p \la}\Big|_{\la=-c^2}$ in greater detail. Explicit calculations 
show that
\begin{align}& \lb{w23}
\frac{\p \psi (z, \la)}{c\, \p \la}=\frac{i}{2}\,\, \frac{1+4iaq(z,\la)e^{2iaq(z,\la)}-e^{4iaq(z,\la)}}{q(z,\la)(z+q(z,\la))^2\left(1-\frac{\la}{(z+q(z,\la))^2}e^{2iq(z,\la)}\right)^2},
\end{align}
and so in view of \eqref{w7} and \eqref{w23} one has 
\begin{align}& \lb{w24}
- \frac{\p \psi (\xi+i\eta, \la)}{c\,  \p \la}\Big|_{\la = - c^2} \in L^1(\bbR; d \xi) \quad (\eta \geq C).
\end{align}

%%%%%%%
\begin{remark}\lb{LR}
According to \cite[Theorem 1.3.6]{Sa13}, \eqref{w24} implies that the distribution functions $\mu(\a,\eta)$ are absolutely continuous
and the densities
\begin{align} & \lb{wf11}
\frac{d}{d \a}\mu(\a,\eta):=\wh \om(\a, \eta)
\end{align}
are continuous (with respect to $\a$). In particular, one infers that inequality \eqref{wf5} is, in fact, an equality:
\begin{align}& \lb{wf12}
 2 \wh \om(\a, \eta)= e^{-\eta \a}\sum_{k=1}^{\infty}k c^{2k} f_k(\a/2) \quad (\eta \geq C).
\end{align}
${}$ \hfill  $\diamond$
\end{remark}
%%%%%%% 

%%%%%%%%%%%%%%%%%%%%%%%%%%%%%%%%%%%%%
\noindent
{\bf Acknowledgments.} 
We are indebted to Mark Ashbaugh for helpful discussions on the material in Part 2 of Appendix \ref{sA}. 
F.G. is indebted to Alexander Sakhnovich for the great hospitality extended to him at the Faculty of Mathematics of the University of Vienna, Austria, during an extended stay in June of 2018. The research of A.L.~Sakhnovich was supported by the Austrian Science Fund (FWF) under Grant No.~P29177. 
%%%%%%%%%%%%%%%%%%%%%%%%%%%%%%%%%%%%%

%%%%%%%%%%%%%%%%
%%%%%%%%%%%%%%%%

\end{document}